\documentclass[leqno,12pt]{article}
\usepackage{amsmath, amssymb}
\usepackage{amsthm} %theorem environment option
\usepackage[all]{xy}
\usepackage[dvips]{graphicx}
\usepackage{color}
\theoremstyle{plain} %text of this environment is typesetted in italics
\newtheorem{theorem}{\indent\sc Theorem}[section]
\newtheorem{lemma}[theorem]{\indent\sc Lemma}
\newtheorem{corollary}[theorem]{\indent\sc Corollary}
\newtheorem{proposition}[theorem]{\indent\sc Proposition}

\newtheorem{conjecture}[theorem]{\indent\sc Conjecture}
\theoremstyle{definition} 
\newtheorem{definition}[theorem]{\indent\sc Definition}
\newtheorem{remark}[theorem]{\indent\sc Remark}
\newtheorem{example}[theorem]{\indent\sc Example}

\makeatletter
\def\address#1#2{\begingroup
\noindent\parbox[t]{7.8cm}{%
\small{\scshape\ignorespaces#1}\par\vskip1ex
\noindent\small{\itshape E-mail address}%
\/: #2\par\vskip4ex}\hfill%
\endgroup}%
\makeatother
\title{\uppercase{Primitive Collections and Toric Varieties}} 
\author{
\bigskip \\
\textsc{David A. Cox and Christine von Renesse} %names of authors
}
\date{} %leave empty
\newcommand{\R}{\mathbb{R}}
\newcommand{\C}{\mathbb{C}}
\newcommand{\Z}{\mathbb{Z}}

\newcommand{\Q}{\mathbb{Q}}

\renewcommand{\P}{\mathbb{P}}
\renewcommand{\/}{/\negthickspace/}

\newcommand{\PL}{\mathrm{PL}}
\newcommand{\CPL}{\mathrm{CPL}}
\newcommand{\pic}{\mathrm{Pic}}
\newcommand{\nef}{\mathrm{Nef}}
\newcommand{\NE}{\mathrm{NE}}

\newcommand{\Span}{\mathrm{span}}

\newcommand\res[1]{{\lower1pt\hbox{$|$}}_{\raise.5pt\hbox{${\scriptstyle #1}$}}}

\def\p{{\rho}}

\def\b{{\beta}}
\def\fan{{\Sigma}}

\def\pos{{\mathrm{Cone}}}

\begin{document}

\maketitle

%%%%%%%%%%%%%%% footnote %%%%%%%%%%%%%%%%
\footnote{ %2000 MSC numbers
2000 \textit{Mathematics Subject Classification}.
Primary 14M25
}
\footnote{ %key words and phrases
\textit{Key words and phrases}.
Toric varieties, primitive collections.
}
%\footnote{ %acknowledgment of support etc. if any
%$^{*}$Thanks.
%}
%%%%%%%%%%%%%%%%%%%%%%%%%%%%%%%%%%%%%%%%%
\begin{abstract}
This paper studies Batyrev's notion of primitive collection.  We use
primitive collections to characterize the nef cone of a
quasi-projective toric variety whose fan has convex support, a result
stated without proof by Batyrev in the smooth projective case.  When
the fan is non-simplicial, we modify the definition of primitive
collection and explain how our definition relates to primitive
collections of simplicial subdivisons.  The paper ends with an open
problem. 
\end{abstract}

\section*{Introduction}

Let $X$ be the the toric variety of a fan $\fan$.  When $X$ is smooth
and projective, Batyrev \cite{batyrev} defines a collection
$\{\p_1,\dots,\p_k\}$ of 1-dimensional cones of $\fan$ to be a
\emph{primitive collection} provided it does not span a cone of
$\fan$ but every proper subset does.  Each primitive collection gives
a \emph{primitive inequality}, and one of the nice results of
\cite{batyrev} states that the nef cone of $X$ is defined by the
primitive inequalities.  For a proof, Batyrev cited the work of Oda
and Park \cite{odapark} and Reid \cite {reid}, without giving details.

The survey article \cite{cox} by the first author notes that Batyrev's
theorem applies to simplicial projective toric varieties.  Casagrande
\cite{casagrande} and Sato \cite{sato} explain how primitive
collections relate to Reid's paper \cite{reid}, and Kresch
\cite{kresch} gives a proof in the smooth case.  However, a complete
proof of Batyrev's result in the simplicial case has never appeared in
print.  In this paper, we give two proofs of Batyrev's theorem, one
based on \cite{kresch} and the other on \cite{reid}.  We also extend
the definition of primitive collection to the non-simplicial case and
show that primitive collections still have the required properties.
Our results apply to all quasi-projective toric varieties whose fans
have convex support of maximal dimension.

\subsection*{Notation} We use standard notation and terminology for
toric varieties.  Let $N$ and $M = \mathrm{Hom}_\Z(N,\Z)$ be dual
lattices of rank $n$ with associated real vector spaces $N_\R =
N\otimes_\Z \R$ and $M_\R = M\otimes_\Z \R$.

Let $X = X_\fan$ be a toric variety of a fan $\fan$ in $N_\R \cong
\R^n$.  We always assume that the support $|\fan|$ of $\fan$ is
convex of dimension $n$.  Hence all maximal cones have dimension $n$.

Given $\fan$, $\fan(k)$ denotes the set of $k$-dimensional cones of
$\fan$, and $\fan(k)^\circ$ is the subset of $\fan(k)$ consisting of
$k$-dimensional cones not lying on the boundary of $|\fan|$.  An
\emph{interior wall} is an element of $\fan(n-1)^\circ$.

We use the convention that $\p$ will denote both an element of
$\fan(1)$ and its primitive generator in $N$.  The torus-invariant divisor
associated to $\rho$ is denoted $D_\rho$.

Also recall that a piecewise-linear function $\phi$ can be represented
by giving $m_{\sigma} \in M_\R$ for each $\sigma \in \fan(n)$, i.e.,
$\phi(u) = \langle m_{\sigma}, u \rangle$ if $u \in \sigma$.  We
define $\PL(\fan)$ as the vector space of all piecewise-linear
functions on $\fan$.  The function $\phi$ is well-defined in
$\PL(\fan)$ if and only if the following statement holds:\ if $\tau$
is an interior wall and $\sigma$, $\sigma'$ are the $n$-dimensional
cones on each side of $\tau$, then $m_{\sigma} - m_{\sigma'} \in
\tau^{\perp}$.  The support function $\phi$ of a torus-invariant
Cartier divisor $D$ satisfies $D = \sum_\rho \phi(\rho) D_\rho$.  Note
the absence of minus signs.

For us, $\phi$ is \emph{convex} if and only $\phi(u) + \phi(v) \ge
\phi(u+v)$ for all $u,v \in |\fan|$.  We also define $\CPL(\fan)
\subset \PL(\fan) $ to be the cone consisting of all convex
piecewise-linear functions on $\fan$.  A function $\phi \in \PL(\fan)$
is \emph{strictly convex} when $\phi(u) + \phi(v) > \phi(u+v)$ for all
$u,v \in |\fan|$ not lying in the same cone of $\fan$.  The toric
variety $X$ is quasi-projective if and only if there exists a strictly
convex $\phi \in \PL(\fan)$.  When this happens, the interior of
$\CPL(\fan)$ is nonempty and consists of all strictly convex
piecewise-linear functions in $\PL(\fan)$.

\subsection*{Outline of the paper}  In Section 1 we give a new
definition of primitive collection and state our main theorem.  We
also recall the nef and Mori cones and review the description of the
Mori cone in terms of the wall relations coming from interior walls.
In Section 2 we prove Batyrev's theorem in the simplicial case, and
then in Section 3 we treat the non-simplicial case.  This section also
studies how primitive collections for $\fan$ relate to primitive
collections for a simplicial subdivision $\fan'$ of $\fan$.  The final
section of the paper explores an open problem dealing with the
quasi-projective hypothesis.

\section{Primitive Collections and the Main Theorem}
\label{section2}

\subsection{Primitive Collections}
The nef cone $\nef(X)$ of $X$ is the quotient of the cone $\CPL(\fan)
\subset \PL(\fan)$ by all linear functions on $\fan$.  Thus
\[
\nef(X) \subset \pic(X)_\R = \pic(X)\otimes_\Z \R. 
\]

Here is the central definition of our paper.

\begin{definition}
\label{pcdef}
A subset $\{ \p_1,\dots, \p_k \} \subset \fan(1)$ is called a
\emph{primitive collection} for $\fan$ if $\{ \p_1,\dots, \p_k \}$ is not
contained in a single cone of $\fan$ but every proper subset is.
\end{definition}

\begin{remark}
In the smooth projective case, Batyrev defined $\{ \p_1,\dots, \p_k \}
\subset \fan(1)$ to be a primitive collection when $\{ \p_1,\dots,
\p_k \}$ does not generate a cone of $\fan$ but every proper
subset does.  When $\fan$ is smooth or more generally simplicial, this
is clearly equivalent to Definition~\ref{pcdef}.
\end{remark}

\begin{definition}
Let $\{ \p_1, \dots, \p_k\}$ be a primitive collection.  We say that
$\phi \in \PL(\fan)$ satisfies the \emph{primitive inequality for}
$\{\p_1, \dots, \p_k\}$ if
\[
\phi(\p_1) + \cdots + \phi(\p_k) \ge \phi(\p_1 + \cdots + \p_k ). 
\]
\end{definition}

If $\phi \in \PL(\fan)$ is convex, i.e., $\phi \in \CPL(\fan)$, then
$\phi$ clearly satisfies the primitive
inequality for every primitive collection.  In other words,
\begin{equation}
\label{cone_inclusion}
\begin{aligned}
\CPL(\Sigma) \subset \big\{& \phi \in \PL(\fan) \mid \phi(\p_1) +
\cdots + \phi(\p_k) \ge \phi(\p_1 + \cdots + \p_k )\\ 
&\text{for all primitive collections}\ \{ \p_1,\dots, \p_k \} \
\text{for} \ \fan\big\}.
\end{aligned}
\end{equation}
The main result of our paper is that the inclusion
\eqref{cone_inclusion} is in fact an equality, i.e., the nef cone is
defined by the primitive inequalities.  Here is the precise statement.

\begin{theorem}[Main Theorem]
\label{mainthm}
Let $X$ be a quasi-projective toric variety coming from the fan
$\fan$ in $N_\R \cong \R^n$.  If $|\fan|$ is convex of
dimension $n$, then
\begin{align*}
\CPL(\Sigma) = \big\{ &\phi \in \PL(\fan) \mid \phi(\p_1) +
\cdots + \phi(\p_k) \ge \phi(\p_1 + \cdots + \p_k )\\ 
&\text{\rm for all primitive collections}\ \{ \p_1,\dots, \p_k \}
\ \text{\rm for}\ \fan \big\}. 
\end{align*}
\end{theorem}

Section 2 will prove Theorem~\ref{mainthm} when $X$ is simplicial and
Section 3 will treat the non-simplicial case.

\subsection{The Mori Cone}
The proof of the main theorem will use extremal rays.  Hence we need
recall the Mori cone of a toric variety.  Although this material is
well-known to experts, we include many details since the results we
need do not appear explicitly in the literature.  We begin with the
exact sequence
\[
0 \longrightarrow M_\R \longrightarrow \PL(\fan) \longrightarrow
\pic(X)_\R \longrightarrow 0
\]
which dualizes to
\[
0 \longrightarrow N_1(X) \longrightarrow  \PL(\fan)^* \longrightarrow
N_\R \longrightarrow 0
\]
where $\PL(\fan)^*$ denotes the dual of $\PL(\fan)$ and $N_1(X)$ is
the dual of $\pic(X)_\R$.  The cone $\CPL(\fan) \subset \PL(\fan)$
contains the image of $M_\R$ and hence dualizes to a cone
\[
\NE(X) = \CPL(\fan)^\vee \subset N_1(X) \subset \PL(\fan)^*.
\]
We call $\NE(X)$ the \emph{Mori cone} of $X$.  When $X$ is
quasi-projective, $\nef(X)$ has maximal dimension in $\pic(X)_\R$, so
that the Mori cone $\NE(X) \subset N_1(X)$ is strongly convex.  The
unique minimal generators of the Mori cone are called \emph{extremal
rays}.

We now review the combinatorial description of $\NE(X)$ in terms of
the interior walls of $\fan$.  The basic observation is that relations
among elements of $\fan(1)$ give elements of $N_1(X)$.  The map
$\phi \in \PL(\fan) \mapsto (\phi(\p))_\p \in \R^{\fan(1)}$ gives a
commutative diagram
\[
\xymatrix{ 0 \ar[r] & M_\R \ar[r] & \R^{\fan(1)} \ar[r] &
A_{n-1}(X)_\R \ar[r] & 0\\ 0 \ar[r] & M_\R \ar[r] \ar[u] & \PL(\fan)
\ar[r] \ar[u] & \pic(X)_\R \ar[r] \ar[u] & 0 }
\]
where $A_{n-1}(X)$ is the Chow group of $(n-1)$-cycles modulo rational
equivalence.  This dualizes to
\[
\xymatrix{
0 \ar[r] & A_{n-1}(X)_\R^* \ar[r] \ar[d] & \R^{\fan(1)*} \ar[r] \ar[d] &
N_\R \ar[r] \ar[d] & 0\\ 
0 \ar[r] & N_1(X) \ar[r]  & \PL(\fan)^* \ar[r] & N_\R
\ar[r] \ar[u] & 0
}
\]
In the top row, the map $\R^{\fan(1)*} \to N_\R$ sends the standard
basis element $e_\rho$ to $\p \in N$.  Thus $A_{n-1}(X)_\R^*$ can be
interpreted as all linear relations among the $\p \in \fan(1)$, and
the surjective map $A_{n-1}(X)_\R^* \to N_1(X)$ shows that all
elements of $N_1(X)$ come from linear relations among the $\p \in
\fan(1)$.

Interior walls of $\fan$ give the following linear relations.
Given an interior wall $\tau$, let $\sigma$ and $\sigma'$ be the
$n$-dimensional cones on each side of $\tau$, i.e., $\tau = \sigma
\cap \sigma'$.  Pick $n-1$ linearly independent vectors $\p_1,\dots,
\p_{n-1}$ in $\tau(1)$ and pick vectors $\p_n \in \sigma(1) \setminus
\tau(1)$, $\p_{n+1} \in \sigma'(1) \setminus \tau(1)$.  Then there is
a nontrivial relation
\begin{equation}
\label{wall_relation}
\sum_{i=1}^{n+1} a_i \p_i =0,\quad a_i \in \Q,\ a_n,a_{n+1} >0,
\end{equation}
where the final condition holds since $\p_n$ and $\p_{n+1}$ lie on
opposite sides of the wall.  Hence the coefficients $a_1,\dots,a_n$
are unique up to multiplication by a positive constant.  Let $a_{\tau}
\in \R^{\fan(1)*}$ have components $(a_{\tau})_{\p_i} = a_i$ for
$i=1, \dots, n+1$ and $(a_{\tau})_\p =0$ otherwise.  Using the above
diagram, we see that $a_{\tau} \in A_1(X)_\R$.

\begin{definition}
\label{ltau_def}
Depending on the context, we use the term \emph{wall relation} to
refer to the equation \eqref{wall_relation}, the vector $a_{\tau} \in
A_{n-1}(X)_\R^*$, or its image $l_\tau \in N_1(X)$.
\end{definition}

Notice that in the non-simplicial case, a given wall can have many
choices for the $\p_1,\dots,\p_{n+1}$ in the wall relation
\eqref{wall_relation}, while in the simplicial case,
$\p_1,\dots,\p_{n+1}$ are uniquely determined by the wall.

\begin{theorem}
\label{thm_walls_mori}
Let $\fan$ be a fan in $N_\R \cong \R^n$ with convex support of
dimension $n$.  
\begin{enumerate}
\item For $\tau \in \fan(n-1)^\circ$, the different choices of the wall
relation \eqref{wall_relation} all give the same $l_{\tau} \in N_1(X)$
up to a positive constant.
\item The Mori cone in $N_1(X)$ is given by
\[
\NE(X)  = \sum_{\tau \in \fan(n-1)^\circ} \R_{\geq 0} l_{\tau}.
\]
\end{enumerate}
\end{theorem}

\begin{proof}
This is the Kleiman-Nakai criterion from Oda and Park \cite[Thm.\
2.3]{odapark}.  We give the details since their definition of $l_\tau$
differs from ours.  

Let $\tau = \sigma \cap \sigma'$ and pick a wall relation
$\sum_{i=1}^{n+1} a_i \p_i =0$, $a_n,a_{n+1} >0$ as in
\eqref{wall_relation}.  Rescaling by a positive constant, we may
assume $a_{n+1} = 1$, so that
\begin{equation}
\label{pnp1eq}
\p_{n+1} = - a_1\p_1 -\cdots -a_n\p_n.
\end{equation}
Then $l_\tau \in \PL(\fan)^*$ is the linear functional on $\PL(\fan)$
given by
\begin{align*}
l_\tau(\phi) &= a_1 \phi(\p_1) + \cdots + a_{n-1} \phi(\p_{n-1}) + a_n
\phi(\p_n) + \phi(\p_{n+1}) \\ 
&= a_1 \langle m_{\sigma'},\p_1\rangle + \cdots + a_{n-1}\langle
m_{\sigma'},\p_{n-1}\rangle + a_n \langle m_{\sigma},\p_n\rangle +
\langle m_{\sigma'},\p_{n+1}\rangle
\end{align*}
since $\phi(u) = \langle m_{\sigma'},u\rangle$ for $u \in \sigma'$ and
$\phi(\p_n) = \langle m_{\sigma},\p_n\rangle$ since $\p_n \in \sigma$.
However, \eqref{pnp1eq} implies that
\begin{align*}
\langle m_{\sigma'},\p_{n+1}\rangle &= \langle m_{\sigma'},- a_1\p_1
-\cdots -a_{n-1} \p_{n-1} -a_n\p_n \rangle\\
&= - a_1\langle m_{\sigma'},\p_1\rangle
-\cdots -a_{n-1}\langle m_{\sigma'}, \p_{n-1}\rangle
-a_n\langle m_{\sigma'}, \p_n \rangle 
\end{align*}
Hence the above formula for $l_\tau(\phi)$ simplifies to
\[
l_\tau(\phi) = a_n \langle m_{\sigma},\p_n\rangle - a_n\langle
m_{\sigma'}, \p_n \rangle  = \langle m_{\sigma}- m_{\sigma'},
a_n\p_n\rangle.
\]
Note that $a_n \p_n \in \sigma \setminus \tau$.  Since $m_\sigma -
m_{\sigma'} \in \tau^\perp$ and $(\sigma + \Span(\tau))/\Span(\tau)$ is
1-dimensional, it follows that up to a positive constant,
\[
l_\tau(\phi) = \langle m_{\sigma}- m_{\sigma'},
v\rangle
\]
for \emph{any} $v \in \sigma \setminus \tau$.  This proves the first
part of the theorem and also shows that our $l_\tau$ agrees with the
$l_\tau$ appearing in the statement of \cite[Thm.\ 2.3]{odapark}.
Then the second part follows immediately from \cite[Thm.\
2.3]{odapark}.
\end{proof}

\subsection{Primitive Relations}
Let $P=\{\p_1,\dots, \p_k\}$ be a primitive collection for $\fan$.
Then $\p_1 + \cdots + \p_k$ lies in some unique minimal cone $\sigma$
of $\Sigma$.  Pick a subset $S \subset \sigma(1)$ satisfying
\begin{equation}
\label{eq_prim_nonsimpl}
\p_1 + \cdots + \p_k = \sum_{\p \in S} b_\p \p, \quad b_\p >
0,\ S  \ \text{linearly independent}.
\end{equation}
The equation \eqref{eq_prim_nonsimpl} gives the vector
$a_{P} \in \R^{\fan(1)*}$ defined by
\[
(a_P)_\p = 
\begin{cases} 
1 & \p \in P\setminus S\\
1-b_\p & \p \in P\cap S\\ 
-b_\p & \p \in S\setminus P\\ 
0 &\text{otherwise}.
\end{cases}
\]
From \eqref{eq_prim_nonsimpl}, it follows that $a_P \in
A_{n-1}(X)_\R^*$.  

\begin{definition}
Depending on the context, we use the term \emph{primitive relation} to
refer to the equation \eqref{eq_prim_nonsimpl}, the vector $a_{P} \in
A_{n-1}(X)_\R^*$, or its image $l_{P} \in N_1(X)$ under the map
$A_{n-1}(X)_\R^* \to N_1(X)$.
\end{definition}

When $X$ is smooth, the case $\p \in P\cap S$ cannot occur
\cite[Prop.\ 3.1]{batyrev}, but happens frequently in the simplicial
case.  In this case, we bound $b_\rho$ as follows.

\begin{lemma}
\label{poscoeffs}
The coefficients in the primitive relation \eqref{eq_prim_nonsimpl}
satisfy $b_\p < 1$ when $\p \in P\cap S$.  Hence $P$ is determined by the
positive entries of $a_P$.
\end{lemma}

\begin{proof}
Suppose $\p_1 \in S$ and $b_{\p_1} \ge 1$.  Subtracting
$\p_1$ from each side of \eqref{eq_prim_nonsimpl} gives
\[
\p_2 + \cdots + \p_k = (b_{\p_1}-1)\p_1 + \sum_{\p \in S
\setminus\{\p_1\}} b_\p \p
\]
Since $\rho_2,\dots,\rho_k$ lie in a cone of $\Sigma$, this equation
implies that $\rho_2,\dots,\rho_k$ and $S \setminus \{\p_1\}$
lie in the same cone of $\Sigma$.  Adding in $\p_1$ shows that $P$
lies in a cone of $\Sigma$, which is impossible.
\end{proof}

The minimal cone $\sigma$ containing $\p_1+\cdots+\p_k$ need not be
simplicial, so there may be many subsets $S$ satisfying
\eqref{eq_prim_nonsimpl}.  But when there are many choices for $a_P$,
they all give the same element $l_{P} \in N_1(X)$, as shown by the
following proposition.

\begin{proposition}
\label{primrel_unique}
Let $P = \{\p_1,\dots,\p_k\}$ be a primitive collection for $\fan$ and
let $l_P \in N_1(X)$ be defined as above.  Then:
\begin{enumerate}
\item When regarded as an element of $\PL(\fan)^*$, $l_P$ is the
linear functional on $\PL(\fan)$ defined by
\[
\phi \longmapsto \phi(\p_1) + \cdots + \phi(\p_k) -
\phi(\p_1 + \cdots + \p_k).
\]
\item $l_P \in \NE(X)$.
\end{enumerate}
\end{proposition}

\begin{proof} Let $\sigma \in \fan$ be the smallest cone
containing $\p_1 + \cdots + \p_k$.  Since $\phi$ is linear on
$\sigma$ and $S \subset \sigma(1)$, we obtain
\begin{align*}
l_P(\phi) &= {\textstyle\sum_\p (a_P)_\p \phi(\p)= \sum_{i=1}^k
\phi(\p_i) - \sum_{\p \in S} b_\p \phi(\p)}\\ 
&= {\textstyle\sum_{i=1}^k \phi(\p_i) - \phi(\sum_{\p \in S}
b_\p \p) = \sum_{i=1}^k \phi(\p_i) - \phi(\sum_{i=1}^k
p_i).}
\end{align*}
This proves the first part of the proposition, and then the second
part follows immediately from the first part and
\eqref{cone_inclusion} since $\NE(X) = \CPL(\fan)^\vee$.
\end{proof}

We can formulate Theorem~\ref{mainthm} in terms of primitive relations
as follows.

\begin{proposition}
\label{equivalent}
Let $X$ be the toric variety of the fan $\fan$ in $N_\R \cong \R^n$.
If $|\fan|$ is convex of dimension $n$, then the following are
equivalent:
\begin{enumerate}
\item $\CPL(\Sigma) = \big\{ \phi \in \PL(\fan) \mid\ \phi(\p_1) +
\cdots + \phi(\p_k) \ge \phi(\p_1 + \cdots + \p_k )$ 
for all primitive collections $\{
\p_1,\dots, \p_k \}\ \text{for}\ \fan \big\}$. 
\item $\NE(X) = \sum_P \R_{\ge 0} l_P$,
where the sum is over all primitive collections for $\fan$.
\end{enumerate}
\end{proposition}

\begin{proof}
This follows easily from
Proposition~\ref{primrel_unique} and $\NE(X) = \CPL(\fan)^\vee$.
\end{proof}

The strategy for proving Theorem~\ref{mainthm} in the simplicial case
is the observation, implicit in \cite{reid}, that every minimial
generator of $\NE(X)$ is a primitive relation $l_P$ for some primitive
collection $P$.  Then Theorem~\ref{mainthm} for simplicial fans
follows immediately from Proposition~\ref{equivalent}.  We
give the details of this argument in Section~2.

\subsection{Curves and the Mori Cone}
An interior wall $\tau$ gives a complete torus-invariant curve
$V(\tau) \cong \P^1$ in $X$.  Let
\[
c_\tau : \pic(X)_\R \longrightarrow \R
\]
denote the linear functional that sends an $\R$-Cartier divisor $D$ to the
intersection product $V(\tau) \cdot D$.  Thus
\[
c_\tau \in \pic(X)_R^* = N_1(X).
\]
Up to a positive multiple, this gives the same class as the wall
relation $l_\tau \in N_1(X)$ from Definition~\ref{ltau_def}.  Although
this result is well-known to experts, we include a proof for
completeness.

\begin{proposition}
\label{ctaultau}
Let $\fan$ be a fan in $N_\R \cong \R^n$ with convex support of
dimension $n$.  For each $\tau \in \fan(n-1)^\circ$, $c_\tau \in
N_1(X)$ is a positive multiple of the wall relation
$l_\tau$ appearing in Theorem~\ref{thm_walls_mori}.
\end{proposition}

\begin{proof}
When $\fan$ is simplicial, we have $\tau(1) = \{\p_1,\dots,\p_{n-1}\}$
and as in the proof of Theorem~\ref{thm_walls_mori}, we have the wall
relation
\[
a_1 \p_1 + \cdots + a_n \p_n + \p_{n+1} = 0.
\]
Since $\fan$ is simplicial, the divisors $D_\p$ corresponding to $\p \in
\fan(1)$ are $\Q$-Cartier, so that $V(\tau) \cdot D_\p$ is defined.  
By \cite[(2.7)]{reid}, we have
\begin{equation}
\label{reid_prop_2.7}
V(\tau) \cdot D_\p = \begin{cases} 0 & \p \notin
  \{\p_1,\dots,\p_{n+1}\}\\ 
a_i V(\tau) \cdot D_{\p_{n+1}} & \p = \p_i,\ i \in \{1,\dots,n\}\\
V(\tau) \cdot D_{\p_{n+1}} > 0 & \p = \p_{n+1}.\end{cases}
\end{equation}
The proof in \cite{reid} assumes $\fan$ is simplicial and complete and
$\tau$ is any wall; the argument applies without change when $\fan$ is
simplicial and $\tau$ is an interior wall.

For the general case, we use the well-known fact that $\fan$ has a
simplicial refinement $\fan'$ such that $\fan'(1) = \fan(1)$ (see
Corollary~\ref{cor_simpl_ref_general} and
Remark~\ref{rem_simpl_ref_general}).  If $X'$ is the toric variety of
$\fan'$, then we have a proper map
\[
X' \overset{\pi}\rightarrow X.
\]
Let $\tau'$ be an interior wall of $\fan'$ contained in $\tau$, and
let $V(\tau')$ and $V(\tau)$ be the corresponding curves in $X'$ and
$X$.  The induced map $\pi\res{V(\tau')} : V(\tau') \to V(\tau)$ has
degree $d = [\Z(\tau\cap N):\Z(\tau'\cap N)]$, which implies that
$\pi_{*} V(\tau')= d\,V(\tau)$.  Let $D$ be a Cartier divisor on
$X$. By the projection formula,
\begin{equation*}
V(\tau) \cdot D = {\textstyle\frac1d}\pi_{*} V(\tau') \cdot D =
{\textstyle\frac1d}V(\tau') \cdot \pi^*D.
\end{equation*}
If we write $D = \sum_{\p} \alpha_\p D_\p$ on $X$, then $\pi^*D = \sum_{\p}
\alpha_\p D_\p$ on $X'$ since $\fan'(1) = \fan(1)$.  If $a_{\tau'} =
(a_\p)_\p$ is the wall relation of $\tau'$ coming from
\eqref{wall_relation}, then up to a positive constant,
\[
V(\tau') \cdot \pi^*D = \sum_\p a_\p \alpha_\p
\]
since $\fan'$ is simplicial.  However, the wall relation for $\tau'$
is one of the (possibly many) wall relations for $\tau$, i.e.,
$a_{\tau'}$ is one of the possible choices for $a_{\tau}$.  Then the
formula
\[
V(\tau) \cdot D = {\textstyle\frac1d}V(\tau') \cdot \pi^*D =
{\textstyle\frac1d}\sum_\p a_\p \alpha_\p
\]
shows (again up to a positive constant) that the class of $V(\tau)$ in
$N_1(X)$ is the image of $a_\tau$ in $N_1(X)$.  In other words,
$c_\tau$ equals $l_\tau$ up to a positive constant, as claimed.
\end{proof}

We conclude our discussion of the Mori cone explaining how our
definition of $\NE(X)$ relates to the standard geometric approach.
Since $X$ need not be complete, we work in the relative context.  Let
$U$ be the affine toric variety of the strongly convex cone
$|\fan|/(|\fan| \cap (-|\fan|))$.  This gives a proper toric morphism
$X \to U$.  

Following \cite{matsuki} or \cite{reid}, the Mori cone of $X \to U$ is
defined as follows.  Let $Z_1(X/U)$ be the free group generated by
irreducible curves in $X$ that map to a point in $U$.  Then we have a
natural pairing
\[
Z_1(X/U) \times \pic(X) \longrightarrow \Z.
\]
By restricting to torus-invariant curves coming from interior walls,
one sees easily that this pairing is nondegenerate with respect to
$\pic(X)$, i.e., if a Cartier divisor satisfies $C\cdot D = 0$ for all
torus-invariant curves $C$ coming from interior walls, then $[D] = 0$
in $\pic(X)$.  It follows that the above pairing induces a perfect
pairing
\[
N_1(X/U) \times \pic(X)_\R \longrightarrow \R.
\]
Thus $N_1(X/U)$ is what we call $N_1(X)$.  Dropping the $U$ from the
notation is reasonable since in our situation $U$ is determined
functorially by $X$.

Finally, $\NE(X/U) \subset N_1(X)$ is the cone generated by
irreducible curves in $X$ that map to a point in $U$, and the
\emph{Mori cone} is its closure $\overline{\NE}(X/U)$ in $N_1(X)$.
Then Theorem~\ref{thm_walls_mori} and Proposition~\ref{ctaultau}
easily imply that
\[
\NE(X) = \NE(X/U) = \overline{\NE(X/U)} = \sum_{\tau \in \fan(n-1)^\circ}
\R_{\ge 0}c_\tau,
\]
where $c_\tau \in N_1(X)$ is the class of the torus-invariant curve
$V(\tau)$ associated to $\tau$.  This is the \emph{Relative Toric Cone
Theorem}.

\begin{remark}
The Relative Toric Cone Theorem is stated for toric morphisms $X \to
S$ by Matsuki \cite[Thm.\ 14-1-4]{matsuki} or Reid \cite[(1.7)]{reid}.
As pointed out by Fujino and Sato in \cite[Ex.\ 4.3]{fs}, this fails
when the torus action on $S$ has no fixed points.  They give the easy
example of the projection map $X = \C^* \times \P^1 \to S = \C^*$.
The fibers of this map are never torus-invariant, so that
torus-invariant curves cannot generate $N_1(X/S)$.  Fortunately, the
Relative Toric Cone Theorem holds for our map $X \to U$ because
$|\fan| \subset N_\R \cong \R^n$ is convex of dimension $n$.
\end{remark}

\section{The Simplicial Case}
\label{section3}

A nice feature of the simplicial case is that $N_1(X) =
A_{n-1}(X)_\R^*$.  Hence an interior wall $\tau$ gives $a_\tau =
l_\tau$, and a primitive collection $P$ gives $a_P = l_P$.

\subsection{Primitive Collections and Extremal Walls}
Let $\tau$ be an interior wall of a simplicial fan
$\fan$ with $\tau(1) = \{\p_1,\dots,\p_{n-1}\}$, and let $\p_n$ and
$\p_{n+1}$ be the generators that are needed to span the cones on each
side of the wall. The uniquely determined wall relation is
\begin{equation}
\label{wallrelation}
\sum_{i=1}^{n+1} a_i \p_i = 0, \ a_n > 0, \ a_{n+1} =1, \ a_i \in \Q,
\end{equation}
by the discussion following \eqref{wall_relation}.  

\begin{proposition}
\label{reid_primcoll}
Let $\fan$ be a quasi-projective simplicial fan with convex support of
dimension $n$.  Let $\tau$ be an extremal wall, meaning
that the wall relation \eqref{wallrelation}
generates an extremal ray $l_\tau \in NE(X)$.  Then:
\begin{enumerate}
\item $P = \{\p_i \mid a_i > 0\}$ 
is a primitive collection for $\fan$.
\item In $\R^{\fan(1)*}$, the primitive relation $a_P$ of $P$ and wall
relation $a_\tau$ of $\tau$ are equal up to a positive constant.
\end{enumerate}
\end{proposition}

\begin{proof} We give two proofs.  The first is
based on \cite[Thm.\ 2.4]{kresch}, which assumes that $X$ is
smooth and complete.  We adapt the argument to the simplicial case as follows.

We first make a useful observation about convex support functions.  If
$\phi$ is convex and $\sigma \in \Sigma$, we can change $\phi$ by a
linear function so that
\begin{equation}
\label{nefdivisor}
\phi(\rho) = 0,\ \rho \in \sigma(1)\ \ \text{and}\ \ \phi(\rho) \ge 0,\
  \rho \notin \sigma(1).
\end{equation}

Now take an extremal wall $\tau$ with wall relation
\eqref{wallrelation}.  Consider the set
\[
P = \{\rho_i \mid a_i > 0\} = \{\rho \mid (a_\tau)_\rho > 0\}.
\]
We will prove that $P$ is a primitive collection whose primitive
relation $l_P$ equals $l_\tau$ up to a positive constant.  Recall that
$a_P = l_P$ and $a_\tau = l_\tau$ since $\Sigma$ is simplicial.

We first prove by contradiction that $P \not\subseteq \sigma(1)$ for
all $\sigma \in \fan$.  Suppose $P \subseteq \sigma(1)$ and take a
strictly convex support function $\phi$.  We may assume that $\phi$ is
of the form \eqref{nefdivisor}.  Since $\phi(\rho) = 0$ for $\rho
\in \Sigma(1)$, we have
\[
\ell_\tau(\phi) = \sum_{\rho \notin \sigma(1)} (a_\tau)_\rho \phi(\rho).
\]
However, $\phi(\rho) \ge 0$ by \eqref{nefdivisor}, and $P \subseteq
\sigma(1)$ implies $(a_\tau)_\rho \le 0$ for $\rho \notin \sigma(1)$.
It follows that $\ell_\tau(\phi) \le 0$, which is impossible since
$\phi$ is strictly convex and $a_\tau = l_\tau \in \NE(X)
\setminus\{0\}$.  Thus no cone of $\Sigma$ contains all rays in $P$.

It follows that some subset $Q \subseteq P$ is a primitive collection.
This gives the primitive relation $a_Q = l_Q \in N_1(X)$, and we also
have $a_\tau \in N_1(X)$.  Let
\[
\beta = a_\tau - \lambda a_Q \in N_1(X),
\]
where $\lambda > 0$.  We claim that if $\lambda$ is sufficiently
small, then
\begin{equation}
\label{betasubset}
\{\rho \mid \b_\rho < 0\} \subseteq \{\rho \mid (a_\tau)_\rho < 0\}.
\end{equation}
To prove this, suppose that $\beta_\rho < 0$ and $(a_\tau)_\rho \ge
0$.  Then the definition of $\beta$ forces $(a_Q)_\rho > 0$, so that
$\rho \in Q$ by Lemma~\ref{poscoeffs}.  Combined with $Q \subseteq P$,
we see that $(a_\tau)_\rho > 0$ by the definition of
$P$.  But we can clearly choose $\lambda$ sufficiently small so that
\[
(a_\tau)_\rho > \lambda (a_Q)_\rho \quad\text{whenever }
(a_\tau)_\rho > 0.
\]
This inequality and the definition of $\beta$ imply $\beta_\rho > 0$,
a contradiction.

We next claim that $\beta \in \NE(X)$.  By \eqref{betasubset}, we
have
\[
\{\rho \mid \beta_\rho < 0\} \subseteq \{\rho \mid (a_\tau)_\rho < 0\}
\subseteq \tau(1),
\]
where the second inclusion follows from \eqref{wallrelation} and the
definition of $a_\tau$.  Now let $\phi$ be convex.  By
\eqref{nefdivisor} with $\sigma = \tau$, we may assume
\[
\phi(\rho) = 0,\ \rho \in
  \tau(1)\ \ \text{and}\ \ \phi(\rho) \ge 0,\ \rho \notin \tau(1).
\]
Then
\[
\beta(\phi) = \sum_{\rho \notin \tau(1)} \phi(\rho) \beta_\rho \ge 0,
\]
where the final inequality follows since $\phi(\rho) \ge 0$ and
$\beta_\rho < 0$ can only happen only when $\rho \in
\tau(1)$.  This proves that $\beta \in \NE(X)$.

Since $a_Q = l_Q \in \NE(X)$ by Proposition~\ref{primrel_unique}, the
equation
\[
a_\tau = \lambda a_\tau  + \beta
\]
expresses $a_\tau = l_\tau$ as a sum of elements of $\NE(X)$.  But
$l_\tau$ is extremal.  This forces $a_Q$ and $\beta$ to lie in the ray
generated by $a_\tau$.  Since $a_Q$ is nonzero, $a_\tau$ is a positive
multiple of $a_Q$.  In particular, they have the same positive
entries.  Then $P = Q$ follows from the definition of $P$ and
Lemma~\ref{poscoeffs}.  This completes the first proof.

The second proof begins with the extremal wall relation
\eqref{wallrelation}.  For $i = 1,\dots,n+1$ set
\[
\Delta_i = \pos(\p_1,...,\widehat{\p_i},...,\p_{n+1}).
\]
In \cite{reid}, Reid proves that
\begin{equation}
\label{reid1}
\bigcup_{a_i > 0} \Delta_i = \pos(\p_i \mid i = 1,\dots,n+1)
\end{equation}
(see the lemma on \cite[p.\ 403]{reid}) and 
\begin{equation}
\label{reid2}
\Delta_i \in \fan(n)\ \text{whenever } a_i > 0
\end{equation}
(see \eqref{reid_prop_2.7} and \cite[Cor.\ 2.10]{reid}).  Reid assumes
that $\fan$ is simplicial and complete.  His proofs generalize to our
situation without change---see \cite{vonRenesse}.

Let $I = \{i \in \{1,\dots,n+1\} \mid a_i > 0\}$, so that $P = \{\p_i
\mid i \in I\}$.  In order to prove that $P$ is a primitive
collection, we first show that $\pos( \p_i \mid i\in I)$ is not a cone
in $\fan$.  So assume $\pos( \p_i \mid i\in I) \in \fan$ and consider
the relation
\[
\sum_{i\in I} a_i \p_i = \sum_{i \in I^c} -a_i \p_{i},
\]
where the coefficients on the left are positive.  Then $\sum_{i\in I}
a_i \p_i$ lies in the relative interior of the cone $\pos( \p_i \mid
i\in I ) \in \fan$, but $\sum_{i \in I^c} -a_i \p_{i}$ lies in the
wall $\tau \in \fan$ since $n,n+1 \in I$ and $a_i \le 0$ for $i \in
I^c$. It follows that $\pos( \p_i \mid i\in I ) \subset \tau$, which
is a contradiction since $\p_n$ and $\p_{n+1}$ do not lie in the wall.

Now we show that every proper subset of $P$ generates a cone of
$\fan$.  Let $K$ be any proper subset of $I$. Then $\pos( \p_i \mid
i\in K )$ is a face of $\Delta_j$ for any $j \in I\setminus K$.  But
$\Delta_j \in \fan$ by \eqref{reid2}. Hence $P= \{\p_i
\mid i \in I\}$ is a primitive collection.

Finally, we consider the primitive relation of $P$, which can be written
\begin{equation}
\label{P_relation}
\sum_{i\in I} \p_i = \sum_{\p \in \sigma(1)} b_\p \p,
\end{equation}
where $\sigma$ is the minimal cone of $\fan$ containing $\sum_{i \in
I} \p_i$.  Since $\pos(\p_1,\dots,\p_{n+1}) = \bigcup_{i\in I}
\Delta_i$ and $\Delta_i \in \fan$ by \eqref{reid1} and \eqref{reid2},
it follows that \eqref{P_relation} is a relation among
$\p_1,\dots,\p_{n+1}$.  This relation is unique up to a nonzero
constant since $\Sigma$ is simplicial.  Thus $a_P$ is a nonzero
multiple of $a_\tau$, necessarily positive by Lemma~\ref{poscoeffs}
and the definition of $P$.
\end{proof}

\begin{remark}
Batyrev clearly knew this proposition, though it is not stated
explicitly in \cite{batyrev}.  Proposition~\ref{reid_primcoll} is
closely related to Theorem 1.5 in Casagrande's paper \cite{casagrande}
and appears implicitly in the remarks preceeding Proposition 2.2 in
Sato's paper \cite{sato}.
\end{remark}

\subsection{The Main Theorem}
We can now prove the simplicial case of our main theorem.

\begin{theorem}
\label{theorem_mori_primitive}
Let $\fan$ be a simplicial quasi-projective fan in $N_\R \cong \R^n$
with convex support of dimension $n$.  Then the cone $\CPL(\fan)$ is
defined by the primitive inequalities, i.e.,
\begin{align*}
\CPL(\fan) = \big\{& \phi \in \PL(\fan) \mid \phi(\p_1 + \cdots + \p_k
) \le \phi(\p_1) + \cdots + \phi(\p_k) \\ &\text{\rm for all
primitive collections}\ \{ \p_1,\dots, \p_k\}\ \text{for}\ \fan\big\}.
\end{align*}
\end{theorem}

\begin{proof}
By Proposition~\ref{equivalent} it suffices to show that the primitive
relations $l_P$ generate the Mori cone.  We already know that $l_P \in
\NE(X)$ (Proposition~\ref{primrel_unique}) and that $\NE(X)$ is
generated by the extremal wall relations $l_\tau$ (Theorem
\ref{thm_walls_mori}).  Furthermore, $\NE(X)$ is generated by extremal
wall relations since $X$ is quasi-projective. Hence it suffices to
show that every extremal wall relation is a primitive relation.  This
is what we proved in Corollary~\ref{reid_primcoll}, and the theorem
follows.
\end{proof}

Here is an example of Theorem~\ref{theorem_mori_primitive}.

\begin{example}
\label{ex_simplicial}
Figure \ref{figure:counterex_def_h} shows the complete simplicial fan
$\fan$ in $\R^3$ with five minimal cone generators:
\[
\p_0 = (0,0,-1), \ \p_1 = (1,1,1), \ \p_2= (1,-1,1), \
\p_3=(-1,-1,1), \ \p_4=(-1,1,1)
\]
and six maximal cones: 
\begin{align*}
&\sigma_1
= \pos(\p_0, \p_1, \p_2 ),\ \sigma_2 = \pos(\p_0, \p_2, \p_3 ),\
\sigma_3 = \pos( \p_0, \p_3, \p_4 ),\\ 
&\sigma_4 = \pos(\p_0, \p_4, \p_1 ),\ \sigma_5 = \pos(\p_1, \p_2,
\p_4 ),\ \sigma_6 = \pos(\p_2, \p_3, \p_4 ).
\end{align*}
\begin{figure}[htb]
\begin{center}
\begin{picture}(0,0)%
\includegraphics{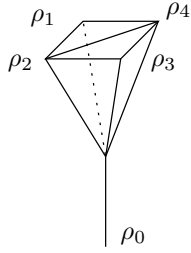}%
\end{picture}%
\setlength{\unitlength}{4144sp}%
\begingroup\makeatletter\ifx\SetFigFont\undefined%
\gdef\SetFigFont#1#2#3#4#5{%
  \reset@font\fontsize{#1}{#2pt}%
  \fontfamily{#3}\fontseries{#4}\fontshape{#5}%
  \selectfont}%
\fi\endgroup%
\begin{picture}(945,1527)(1801,-2098)
\put(2746,-691){\makebox(0,0)[lb]{\smash{{\SetFigFont{10}{14.4}{\familydefault}{\mddefault}{\updefault}{\color[rgb]{0,0,0}$\rho_4$}%
}}}}
\put(2476,-2041){\makebox(0,0)[lb]{\smash{{\SetFigFont{10}{14.4}{\familydefault}{\mddefault}{\updefault}{\color[rgb]{0,0,0}$\rho_0$}%
}}}}
\put(2656,-1006){\makebox(0,0)[lb]{\smash{{\SetFigFont{10}{14.4}{\familydefault}{\mddefault}{\updefault}{\color[rgb]{0,0,0}$\rho_3$}%
}}}}
\put(1801,-1006){\makebox(0,0)[lb]{\smash{{\SetFigFont{10}{14.4}{\familydefault}{\mddefault}{\updefault}{\color[rgb]{0,0,0}$\rho_2$}%
}}}}
\put(1936,-736){\makebox(0,0)[lb]{\smash{{\SetFigFont{10}{14.4}{\familydefault}{\mddefault}{\updefault}{\color[rgb]{0,0,0}$\rho_1$}%
}}}}
\end{picture}%
\end{center}
\caption{Simplicial Fan in $\R^3$}
\label{figure:counterex_def_h}
\end{figure}

The primitive collections for this fan are:
\[
P_1 = \{ \p_1, \p_3\},\quad P_2  =\{ \p_0, \p_2, \p_4\},
\]
so that $\phi \in \PL(\fan)$ is convex if and only if 
\begin{align*}
\phi(\p_1) + \phi(\p_3) &\ge \phi(\p_1+ \p_3)\\
\phi(\p_0) + \phi(\p_2) + \phi(\p_4) &\ge \phi(\p_0+\p_2+ \p_4).
\end{align*}

To get a more concrete characterization, we use the associated
primitive relations:
\begin{align*}
P_1: \ &\p_1 + \p_3 = \p_2 + \p_4\\
P_2: \ &\p_0 + \p_2+ \p_4 =
{\textstyle\frac{1}{2}}\p_2 + \textstyle{\frac{1}{2} }\p_4.
\end{align*}
Let $D_i$ be the torus-invariant divisor associated to $\p_i$.  Then
the divisor $D = \sum_{i=0}^4 a_i D_i$ corresponds to the support
function $\phi$ satisfying $\phi(\p_i) = a_i$.  It follows that $D$ is
nef if and only if
\begin{align*}
a_1 + a_3 &\ge a_2+ a_4\\
a_0 + a_2 + a_4 &\ge {\textstyle\frac{1}{2}}a_2 +
\textstyle{\frac{1}{2} }a_4, \ \text{i.e.,}\ 2a_0 + a_2 +a_4 \ge 0. 
\end{align*}

In contrast, $\fan$ has 9 walls, so Theorem~\ref{thm_walls_mori}
describes $\NE(X)$ using 9 generators, corresponding to 9 wall
inequalities defining $\CPL(\fan) \subset \PL(\fan)$.  Fortunately,
these can be simplified considerably.  We denote by $\tau_{i,j}$ the
wall that is spanned by $\p_i$ and $\p_j$.  By abuse of notation we
will also call $\tau_{i,j}$ the corresponding class in $\NE(X)$. One
can compute that the 9 walls fall into three groups:
\begin{align*}
&\tau_{2,4}\\
&\tau_{1,2} \equiv \tau_{3,4} \equiv \tau_{2,3} \equiv
\tau_{1,4} \equiv 4 \tau_{0,1} \equiv 4 \tau_{0,3}\\
&\tau_{0,2}
\equiv \tau_{0,4} \equiv 2\tau_{1,2} + 2\tau_{2,4}.
\end{align*}
Hence $\tau_{2,4}$ and $\tau_{1,2} \equiv \cdots \equiv 4\tau_{0,3}$
give the extremal rays of the Mori cone, while $\tau_{0,2} \equiv
\tau_{0,4}$ do not give an extremal ray.  One can check that the
primitive collection $P_1$ generates the same ray as $\tau_{2,4}$ and
$P_2$ generates the same ray as $\tau_{1,2}$. \hfill \qed
\end{example}

Example~\ref{ex_simplicial} is nice because there were few
primitive collections.  However, there are examples where the
primitive collections vastly outnumber the interior walls.

\begin{example}
\label{polygon_example}
Let $\fan$ be a complete fan in $\R^2$ with $r \ge 4$ minimal
generators, say $\p_1,\dots,\p_r$, arranged counterclockwise around
the origin. Then there are $r$ walls, all interior.  One easily checks
that the primitive collections are given by $P = \{\p_i,\p_j\}$ for $i
< j$ and $\p_i, \p_j$ not adjacent.  Hence the fan $\fan$ has
\[
\binom{r}{2} - r = \frac{r(r-3)}2.
\]
primitive collections.  This is greater than the number of walls
provided $r \ge 6$. \hfill \qed
\end{example}

\section{The Non-Simplicial Case}

\subsection{Simplicial Refinements}
In order to prove our main theorem in the non-simplicial case, we need
to consider simplicial refinements.  Here we present a general theorem
about the existence of simplicial refinements with special properties.

\begin{theorem}
\label{thrm_simpl_ref_general}
Let $\fan$ be a fan in $N_\R \cong \R^n$ with convex support of
dimension $n$ and fix $P \subset \fan(1)$ such that $P \cap
\sigma(1)$ is linearly independent for all $\sigma \in \fan$. Then:
\begin{enumerate}
\item There exists a simplicial refinement $\fan'$ of $\fan$ satisfying
$\fan'(1) = \fan(1)$ such that $P \cap \sigma(1)$ generates a cone
in $\fan'$ for all $\sigma \in \fan$.
\item If in addition $\fan$ is quasi-projective, then the refinement
$\fan'$ in part {\rm (1)} can be chosen to be quasi-projective.
\end{enumerate}
\end{theorem}

\begin{proof}
To create $\fan'$, assign a weight $w_\p$ to each $\p \in
\fan(1)$ as follows:
\begin{itemize}
\item For $\p \in P$, set $w_\p = 1$.
\item For $\p \in \fan(1) \setminus P$, pick $0 < w_\p < 1$
  generic.  The exact meaning of generic will be explained in the
  course of the proof.
\end{itemize}

For each cone $\sigma \in \fan(n)$, fix $m_\sigma$ in the interior of
the dual cone $\sigma^\vee \subseteq M_\R$.  The affine hyperplane
$H_\sigma \subset N_\R$ defined by $\langle m_\sigma,-\rangle
= 1$ intersects $\sigma$ in a convex polytope $Q_\sigma = H_\sigma
\cap \sigma$ with vertices $\mathbf{v}_\p^\sigma$, where
$\{\mathbf{v}_\p^\sigma\} = (\R_{\ge0} \p) \cap H_\sigma$ for $\p \in
\sigma(1)$.

The idea is to triangulate $Q_\sigma$ using a variant of the method
used in Example 1.1 of \cite{gkz}, p.\ 215.  Consider
\[
G_{\sigma,w} = \mathrm{Conv}(0,w_\p \mathbf{v}_\p^\sigma \mid \p \in
\sigma(1)).
\]
Since the vectors $w_\p \mathbf{v}_\p^\sigma$ lie on the
1-dimensional rays of $\sigma$, it is easy to see that the vertices of
$G_{\sigma,w}$ consist of the origin and the points $w_\p
\mathbf{v}_\p^\sigma$ for $\p \in \sigma(1)$.  Furthermore, the
faces of $G_{\sigma,w}$ not containing the origin project to a
polyhedral subdivision of $Q_\sigma$.  Projecting from the origin in
$N_\R$, we get a refinement $\fan_{\sigma}$ of $\sigma$ that
satisfies $\fan_{\sigma}(1) = \sigma(1)$.  Figure
\ref{figure:simplicial_proof_weights} shows two 3-dimensional cones
$\sigma$, each with a set $P\cap\sigma(1)$ and a choice of weights
giving the polytope $G_{\sigma,w}$ inside $\sigma$.
\begin{figure}[hbt]
\begin{center}
\begin{picture}(0,0)%
\includegraphics{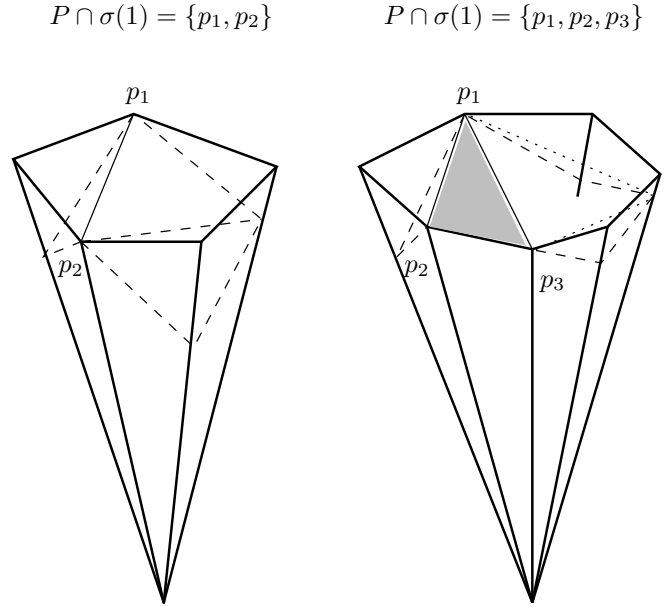}%
\end{picture}%
\setlength{\unitlength}{4144sp}%
\begingroup\makeatletter\ifx\SetFigFont\undefined%
\gdef\SetFigFont#1#2#3#4#5{%
  \reset@font\fontsize{#1}{#2pt}%
  \fontfamily{#3}\fontseries{#4}\fontshape{#5}%
  \selectfont}%
\fi\endgroup%
\begin{picture}(3914,3619)(1059,-3008)
\put(1756, 29){\makebox(0,0)[lb]{\smash{{\SetFigFont{10}{14.4}{\familydefault}{\mddefault}{\updefault}{\color[rgb]{0,0,0}$p_1$}%
}}}}
\put(1351,-1006){\makebox(0,0)[lb]{\smash{{\SetFigFont{10}{14.4}{\familydefault}{\mddefault}{\updefault}{\color[rgb]{0,0,0}$p_2$}%
}}}}
\put(1296,479){\makebox(0,0)[lb]{\smash{{\SetFigFont{10}{14.4}{\familydefault}{\mddefault}{\updefault}{\color[rgb]{0,0,0}$P\cap\sigma(1)=\{p_1, p_2\}$}%
}}}}
\put(3736, 29){\makebox(0,0)[lb]{\smash{{\SetFigFont{10}{14.4}{\familydefault}{\mddefault}{\updefault}{\color[rgb]{0,0,0}$p_1$}%
}}}}
\put(3421,-1006){\makebox(0,0)[lb]{\smash{{\SetFigFont{10}{14.4}{\familydefault}{\mddefault}{\updefault}{\color[rgb]{0,0,0}$p_2$}%
}}}}
\put(4231,-1096){\makebox(0,0)[lb]{\smash{{\SetFigFont{10}{14.4}{\familydefault}{\mddefault}{\updefault}{\color[rgb]{0,0,0}$p_3$}%
}}}}
\put(3296,479){\makebox(0,0)[lb]{\smash{{\SetFigFont{10}{14.4}{\familydefault}{\mddefault}{\updefault}{\color[rgb]{0,0,0}$P\cap\sigma(1)=\{p_1, p_2, p_3\}$}%
}}}}
\end{picture}%
\caption{Two examples of $\sigma$, $P\cap\sigma(1)$, and a choice of weights.}
\label{figure:simplicial_proof_weights}
\end{center}
\end{figure}

We claim that the fans
$\fan_{\sigma}$ have the following three properties:
\begin{itemize}
\item[{A}.] $P\cap \sigma(1)$ generates a cone of
$\fan_{\sigma}$ for all $\sigma \in \fan$.
\item[{B}.] If a face $\tau$ lies in maximal cones $\sigma$ and
$\sigma'$, then $\fan_{\sigma}$ and $\fan_{\sigma'}$ induce the
same refinement of $\tau$.
\item[{C}.] If the $w_\p$ are sufficiently generic for $\p \in
\fan(1) \setminus P$, then $\fan_{\sigma}$ is simplicial for
all $\sigma \in \fan$.
\end{itemize}

Assuming {A}--{C}, the set
\[
\fan' = \bigcup_{\sigma \in \fan} \fan_{\sigma}
\]
is a fan that refines $\fan$ by {B} and satisfies $\fan'(1) =
\fan(1)$.  Furthermore, $\fan'$ is simplicial by {C}.  Finally,
given $\sigma \in \fan$, $P\cap \sigma(1)$ generates a cone of
$\fan'$ by {A}. Hence the proof of part (1) of the theorem will be
complete once we prove {A}--{C}.

\emph{Proof of {A}}.  Consider the hyperplane $H_\sigma \subset N_\R$.
Since $w_\p \le 1$ for all $\p$, the polytope $G_{\sigma,w}$ lies
on the side of $H_\sigma$ containing the origin, and the intersection
$H_\sigma \cap G_{\sigma,w}$ is clearly the convex hull of the points
$\mathbf{v}_\p^\sigma$ for $\p \in P\cap\sigma(1)$ by the choice
of the weights $w_\p$.  It follows that $P\cap\sigma(1)$ generates a
cone of $\fan_{\sigma}$.

\emph{Proof of {B}}.  Suppose $\tau$ is a face of $\sigma$.  Set
\[
G_{\tau,\sigma,w} = \mathrm{Conv}(0,w_\p \mathbf{v}_\p^\sigma \mid
\p \in \tau(1))
\]
and observe that
\[
G_{\tau,\sigma,w} = G_{\sigma,w}\cap \tau.
\]
This tells us that the refinement of $\tau$ induced by $\fan_\sigma$
is determined entirely by the vectors $w_\p \mathbf{v}_\p^\sigma$
for $\p \in \tau(1)$.  Since we are working with fixed weights in
this part of the proof, the only choice involved is this refinement of
$\tau$ comes from the vectors $\mathbf{v}_\p^\sigma$ for $\p \in
\tau(1)$.  Recall that $\{\mathbf{v}_\p^\sigma\} = \p \cap
H_\sigma$, where $H_\sigma$ is defined by $\langle m_\sigma,-\rangle =
1$ and $m_\sigma \in \sigma^\vee$.

Now suppose that $\tau$ is a face of $\sigma'$ for a maximal cone
$\sigma' \ne \sigma$.  The refinement of $\tau$ induced by
$\fan_{\sigma'}$ is determined by
\[
G_{\tau,\sigma',w} = \mathrm{Conv}(0,w_\p
\mathbf{v}_\p^{\sigma'} \mid \p \in \tau(1)),
\]
where $m_{\sigma'} \in {\sigma'}^\vee$ determines the affine
hyperplane $H_{\sigma'}$ whose intersection with $\p \in \tau(1)$
determines $\mathbf{v}_\p^{\sigma'}$.  

Now consider the map $\varphi : \tau \to \tau$ defined as follows:
\begin{itemize}
\item If $\mathbf{v} \in H_\sigma\cap \tau$, then $\{\varphi(\mathbf{v})\}
  = \mathbb{R}_{\ge 0}\mathbf{v} \cap H_{\sigma'}$.
\item If $\mathbf{v} \in H_\sigma\cap \tau$ and $\lambda \ge 0$, then
$\varphi(\lambda \mathbf{v}) = \lambda\varphi(\mathbf{v})$.
\end{itemize}
The map $\varphi$ sends line segments to line segments and is
homogeneous of degree $1$ for non-negative scalars.  Furthermore,
$\varphi(\mathbf{v}_\p^\sigma) = \mathbf{v}_\p^{\sigma'}$, and thus
$\varphi(w_\p\mathbf{v}_\p^\sigma) = w_\p
\mathbf{v}_\p^{\sigma'}$ for all $\p \in \tau(1)$.  It follows that
\[
\varphi(G_{\tau,\sigma,w}) = G_{\tau,\sigma',w}.
\]
This map preserves faces, so that $G_{\tau,\sigma,w}$ and
$G_{\tau,\sigma',w}$ induce the same refinement of $\tau$.  This
completes the proof of {B}.

\emph{Proof of {C}}.  We may assume $\sigma \in \Sigma(n)$.  For $\p
\in \sigma(1)$, write 
\[
\mathbf{v}_\p^\sigma = (a^\p_1,\dots,a^\p_n) \in \mathbb{R}^n.
\]
When we have $\mathbf{v}_{\p_1}^\sigma,
\mathbf{v}_{\p_{2}}^\sigma$, \dots, we write instead
\[
\mathbf{v}_{\p_i}^\sigma = (a^{(i)}_1,\dots,a^{(i)}_{n})
\]
and we set $w_i = w_{\p_i}$.

Now suppose that $\fan_\sigma$ is non-simplicial for some choice of
weights $w_\p$.  This implies that $G_{\sigma,w}$ has a face $F$ of
dimension $n-1$ not containing the origin that is not an
$(n-1)$-simplex.  It follows that $F$ has at least $n+1$ vertices.
Pick $n+1$ vertices of $F$ as follows.  We first pick those vertices
of $F$ of the form $\mathbf{v}_\p^\sigma =
w_\p\mathbf{v}_\p^\sigma$ for $\p \in P \cap \sigma(1)$.  There
are at most $n$ such vertices since $P \cap \sigma(1)$ is linearly
independent.  Since $\dim(F) = n-1$, we can extend them to affinely
independent vertices $w_1 \mathbf{v}_{\p_1}^\sigma,\dots,w_n
\mathbf{v}_{\p_{n}}^\sigma$.  Since $F$ is non-simplicial, we can
pick one more, $w_{n+1} \mathbf{v}_{\p_{n+1}}^\sigma$.

Now consider the $(n+1)\times (n+1)$ matrix 
\[
\begin{pmatrix} 1 & w_1 a^{(1)}_1 & \cdots & w_1 a^{(1)}_{n}\\
\vdots & \vdots && \vdots \\
1 & w_{n+1}a^{(n+1)}_1 & \cdots & w_{n+1}a^{(n+1)}_{n}\end{pmatrix}.
\]
By construction, the vectors
$w_1\mathbf{v}_{\p_1}^\sigma,\dots,w_n\mathbf{v}_{\p_{n}}^\sigma$
are affinely independent, but once we add the last vector,
$w_1\mathbf{v}_{\p_1}^\sigma,\dots,w_{n+1}\mathbf{v}_{\p_{n+1}}^\sigma$
are affinely dependent since they lie in a $(n-1)$-dimensional face.

It follows that this matrix has rank exactly $n$.  Since the weights
$w_i$ are nonzero, the same is true for the matrix
\[
M = \begin{pmatrix} w_1^{-1} &  a^{(1)}_1 & \cdots & a^{(1)}_{n}\\
\vdots & \vdots && \vdots \\
w_{n+1}^{-1} & a^{(n+1)}_1 & \cdots & a^{(n+1)}_{n}\end{pmatrix}.
\]
The determinant of $M$ must vanish.  The resulting linear equation in
the $w_i^{-1}$ gives a necessary condition for $\fan_\sigma$ to be
non-simplicial.  Furthermore, our careful choice of
$w_1\mathbf{v}_{\p_1}^\sigma, \dots,
w_{n+1}\mathbf{v}_{\p_{n+1}}^\sigma$ guarantees that $w_{n+1}^{-1}$
actually appears in the determinant and that $w_{n+1}^{-1} =
w_\p^{-1}$ for some $\p \in \sigma(1)\setminus P$.

Since there are only finitely many maximal cones in $\fan$ and for
each $\sigma$ we have a fixed choice of the $\mathbf{v}_{\p}^\sigma$
(determined by the fixed choice of $m_\sigma \in \sigma^\vee$), we get
a finite system of non-trivial linear equations in the $w_\p^{-1}$
for $\p \in \sigma(1)\setminus P$ that give necessary conditions for
$\fan_\sigma$ to be non-simplicial.  If we pick the weights $w_\p$
to avoid these finitely many subspaces, the resulting subdivisions
$\fan_\sigma$ will be all simplicial.  This completes the proof of
{C}, and part (1) follows. 

Turning to part (2), we assume that $\fan$ is quasi-projective.  It
suffices to find a simplicial refinement $\fan'$ of $\fan$ satisfying
$\fan'(1) = \fan(1)$ such that the induced map $X_{\fan'} \to X_\fan =
X$ is projective.  The latter happens when $\fan'$ has a
piecewise-linear function $\varphi \in \PL(\fan')$ which is strictly
convex relative to $\fan$, meaning that for all $\sigma \in \fan$,
$\varphi\res{\sigma}$ is strictly convex with respect to the subfan
$\{\sigma' \in \fan' \mid \sigma' \subset \sigma\}$ (see (*) on page
27 and Theorem 10 on pages 31--32 of \cite{kkms}).

Since $\fan$ is quasi-projective, we can find $\phi \in \CPL(\fan)$
which is strictly convex.  We first modify $\phi$ so that it
takes positive values on $\fan(1)$.  To see why this is possible, 
consider the cone 
\[
\widehat\sigma = \pos((\p,\phi(\p)) \mid \p
\in \fan(1)) \subset N_\R \times \R.
\] 
Since $\phi$ is strictly convex for $\fan$, it follows that
$\widehat\sigma$ is a strongly convex cone with minimal generators
given by $(\p,\phi(\p))$ for $\p \in \fan(1)$.  Hence we can find
$(m,\mu) \in M_\R \times \R$ such that 
\[
\langle m,\p\rangle + \mu\phi(\p) > 0 \ \text{for all}\ \p \in
\fan(1).
\]
Replacing $\phi$ with $\langle m,-\rangle + \mu\phi$, we may assume
$\phi(\p) > 0$ for all $\p \in \fan(1)$, as claimed.

The proof of part (1) used a hyperplane $H_\sigma \subset N_\R$ for each
$\sigma \in \fan(n)$.  More precisely, we picked $m_\sigma \in
\sigma^\vee$ such that $H_\sigma = \{ u \in N_\R \mid \langle
m_\sigma,u\rangle = 1\}$, and then for $\p \in \sigma(1)$,
$\mathbf{v}_\p^\sigma$ was the unique vector in $\R_{\ge 0}\p$
satisfying $\langle m_\sigma,\mathbf{v}_\p^\sigma\rangle = 1$.

Using $\phi$, we get a consistent set of hyperplanes since
$\phi\res{\sigma}$ is linear, i.e., $\phi(u) = \langle
m_\sigma,u\rangle$ for some $m_\sigma \in M_\R$.  Our hypothesis that
$\phi(\p) > 0$ for all $\p \in \fan(1)$ guarantees that $m_\sigma$ is
in the interior of $\sigma^\vee$.  Hence we can use these $m_\sigma$'s
to give the hyperplanes $H_\sigma$.  Then the point
$\mathbf{v}_\p^\sigma$ is the unique vector in $\R_{\ge 0}\p$
satisfying $\phi(\mathbf{v}_\p^\sigma) = 1$.

Now pick generic weights $w_\p$ for $\p \in \fan(1) \setminus P$.
Then the properties A--C are satisfied (note B is now trivial because
of our consistent choice of the $H_\sigma$).  Thus we get a
simiplicial refinement $\fan'$ of $\fan$ that satisifes part (1) of
the theorem.  Now define $\varphi : |\fan'| = |\fan| \to \R$ by
setting
\[
\varphi(\p) = w_\p^{-1}\phi(\p),\quad \p \in \fan'(1) = \fan(1),
\]
and extending linearly on each cone $\sigma' \in \fan'$.  This gives a
well-defined function in $\PL(\fan')$ since $\fan'$ is simplicial.
Assuming $\phi$ is rational, we can also assume that the $w_\p$ are
rational.  Hence we can assume that $\varphi$ is rational as well.

We claim that $\varphi$ is strictly convex with respect to
$\fan_\sigma = \{\sigma' \in \fan' \mid \sigma' \subset \sigma\}$ for
each $\sigma \in \fan$.  To see this, first observe that
\[
\varphi(w_\p\mathbf{v}_\p^\sigma) = 1, \quad \p \in \sigma(1),
\]
since $\phi(\mathbf{v}_\p^\sigma) = 1$.  It follows that inside
$\sigma$, the inequality $\varphi \le 1$ defines
\[
G_{\sigma,w} = \mathrm{Conv}(0,w_\p \mathbf{v}_\p^\sigma \mid \p \in
\sigma(1)).
\]
Then the convexity of $G_{\sigma,w}$ implies that if $u,v \in \sigma$,
then
\[
\varphi(u) + \varphi(v) \ge \varphi(u+v),
\]
with equality if and only if $u,v$ lie in the same cone of
$\fan_\sigma$.  To prove this, we may assume $u,v \ne 0$, so that
\[
u = \lambda u_0,\ v = \mu v_0,\quad \text{where}\ \lambda,\mu > 0 \
\text{and}\ \varphi(u_0) = \varphi(v_0) = 1.
\]
Then $\frac{\lambda}{\lambda+\mu} u_0 + \frac{\mu}{\lambda+\mu} v_0
\in G_{\sigma,w}$, so that
\[
\varphi(u+v) = (\lambda+\mu)\varphi\big({\textstyle
  \frac{\lambda}{\lambda+\mu} u_0 + \frac{\mu}{\lambda+\mu} v_0}\big) \le
  \lambda+\mu =  \varphi(u)+\varphi(v).
\]
It is equally easy to show that equality occurs exactly when $u,v$ lie
in the same cone of $\fan_\sigma$.  Hence $\varphi$
has the required properties, which completes the proof of part~(2) of
the theorem.
\end{proof}

\begin{corollary}
\label{cor_simpl_ref_general}
If $\fan$ is a fan in $N_\R \cong \R^n$ with convex support of
dimension $n$, then there exists a simplicial refinement $\fan'$ with
the same 1-dimensional generators.  Furthermore, if $\fan$ is
quasi-projective, then we can assume that $\fan'$ is also
quasi-projective.
\end{corollary}

\begin{proof}
Apply Theorem \ref{thrm_simpl_ref_general} with $P=\emptyset$.
\end{proof}

\begin{remark}
\label{rem_simpl_ref_general}
This corollary guarantees the existence of simplicial refinements that
introduce no new generators and preserve quasi-projectivity.  This
result has other proofs, including Fujino \cite{fujino} (via the toric
Mori program) and Thompson \cite{thompson} (via stellar subdivision).
\end{remark}

\subsection{The Main Theorem} We can now prove the non-simplicial case of
our main theorem.

\begin{theorem}
\label{theorem_mori_nonsimp}
Let $\fan$ be a non-simplicial quasi-projective fan in $N_\R \cong
\R^n$ with convex support of dimension $n$.  Then the cone
$\CPL(\fan)$ is defined by the primitive inequalities, i.e.,
\begin{align*}
\CPL(\fan) = \big\{& \phi \in \PL(\fan) \mid \phi(\p_1 + \cdots + \p_k
) \le \phi(\p_1) + \cdots + \phi(\p_k) \\ &\text{\rm for all
primitive collections} \ \{ \p_1,\dots, \p_k\} \ \text{for}\ \fan\big\}.
\end{align*}
\end{theorem}

\begin{proof} By Corollary~\ref{cor_simpl_ref_general}, $\fan$ has a
quasi-projective simplicial refinement $\fan'$ satisfying $\fan'(1) =
\fan(1)$.  Then observe that
\[
\CPL(\fan) = \PL(\fan) \cap \CPL(\fan')
\]
and that
\begin{align*}
\CPL(\fan') = \big\{& \phi \in \PL(\fan') \mid \phi(\p_1 + \dots +
\p_k ) \le \phi(\p_1) + \cdots + \phi(\p_k) \\
&\text{\rm for all primitive collections} \ \{ \p_1,\dots,
\p_k\}\ \text{\rm for} \ \fan'\big\}
\end{align*}
since $\fan'$ is simplicial and quasi-projective.  Hence
\begin{equation}
\label{CPL_eq}
\begin{aligned}
\CPL(\fan) = \big\{& \phi \in \PL(\fan) \mid \phi(\p_1 + \cdots + \p_k
) \le \phi(\p_1) + \cdots + \phi(\p_k) \\ &\text{\rm for all
primitive collections} \ \{ \p_1,\dots, \p_k\}\ \text{\rm for} \ \fan'\big\}.
\end{aligned}
\end{equation}

We divide primitive collections $P = \{ \p_1,\dots, \p_k\}$ for $\fan'$
into two types:
\begin{align*}
\text{Type A:}\ &P \subset \sigma(1)\ \text{for some} \ \sigma \in \fan\\
\text{Type B:}\ &P \not\subset \sigma(1)\ \text{for all} \ \sigma \in \fan.
\end{align*}
Note that if $\phi \in \PL(\fan)$, then $\phi(\p_1 + \cdots + \p_k ) =
\phi(\p_1) + \cdots + \phi(\p_k)$ when $P = \{\p_1,\dots,\p_k\}$ is a
Type A primitive collection for $\fan'$.  Hence these can be omitted
in \eqref{CPL_eq}, so that
\begin{equation}
\label{CPL_eq2}
\begin{aligned}
\CPL(\fan) = \big\{& \phi \in \PL(\fan) \mid \phi(\p_1 + \cdots + \p_k
) \le \phi(\p_1) + \cdots + \phi(\p_k) \\ &\text{\rm for all Type B
primitive collections} \ \{ \p_1,\dots, \p_k\}\ \text{\rm for} \
\fan' \big\}.
\end{aligned}
\end{equation}
However, a Type B primitive collection $P$ for $\fan'$ is a primitive
collection for $\fan$.  This is easy to prove.  First, $P$ is not
contained in any cone of $\fan$ by the definition of Type~B, and
second, every proper subset of $P$ is contained in a cone of $\fan'$
and hence lies in a cone of $\fan$ since $\fan'$ refines $\fan$.  It
follows that
\[
\phi(\p_1 + \cdots + \p_k) \le \phi(\p_1) + \cdots + \phi(\p_k)
\]
is a primitive inequality for $\fan$ whenver $P= \{ \p_1,\dots, \p_k\}$
is a Type B primitive collection for $\fan'$.

Hence \eqref{CPL_eq2} shows that a subset of the primitive
inequalities for $\fan$ define $\CPL(\fan)$ inside $\PL(\fan)$.  Using
the inclusion \eqref{cone_inclusion}, the theorem now follows immediately.
\end{proof}

Here is an example to illustrate Theorem~\ref{theorem_mori_nonsimp}
and its proof.

\begin{example}
\label{nonsimpl_ex}
Figure~\ref{figure:counterex_def} shows the complete non-simplicial
fan $\fan$ in $\R^3$ with five minimal generators:
\begin{equation*}
\p_0 = (0,0,-1),\ \p_1 = (1,1,1),\ \p_2= (1,-1,1), \;
\p_3=(-1,-1,1),\ \p_4=(-1,1,1)
\end{equation*}
and five maximal cones:
\begin{align*}
&\sigma_1 = \pos( \p_0, \p_1, \p_2 ), 
\sigma_2 = \pos( \p_0, \p_2, \p_3 ), 
\sigma_3 = \pos( \p_0, \p_3, \p_4 ),\\
&\sigma_4 = \pos( \p_0, \p_4, \p_1 ), 
\sigma_5 = \pos( \p_1, \p_2, \p_3, \p_4 ).
\end{align*}
\begin{figure}[htb]
\begin{center}
\begin{picture}(0,0)%
\includegraphics{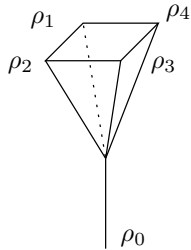}%
\end{picture}%
\setlength{\unitlength}{4144sp}%
\begingroup\makeatletter\ifx\SetFigFont\undefined%
\gdef\SetFigFont#1#2#3#4#5{%
  \reset@font\fontsize{#1}{#2pt}%
  \fontfamily{#3}\fontseries{#4}\fontshape{#5}%
  \selectfont}%
\fi\endgroup%
\begin{picture}(945,1527)(1801,-2098)
\put(2746,-691){\makebox(0,0)[lb]{\smash{{\SetFigFont{10}{14.4}{\familydefault}{\mddefault}{\updefault}{\color[rgb]{0,0,0}$\rho_4$}%
}}}}
\put(2476,-2041){\makebox(0,0)[lb]{\smash{{\SetFigFont{10}{14.4}{\familydefault}{\mddefault}{\updefault}{\color[rgb]{0,0,0}$\rho_0$}%
}}}}
\put(2656,-1006){\makebox(0,0)[lb]{\smash{{\SetFigFont{10}{14.4}{\familydefault}{\mddefault}{\updefault}{\color[rgb]{0,0,0}$\rho_3$}%
}}}}
\put(1801,-1006){\makebox(0,0)[lb]{\smash{{\SetFigFont{10}{14.4}{\familydefault}{\mddefault}{\updefault}{\color[rgb]{0,0,0}$\rho_2$}%
}}}}
\put(1936,-736){\makebox(0,0)[lb]{\smash{{\SetFigFont{10}{14.4}{\familydefault}{\mddefault}{\updefault}{\color[rgb]{0,0,0}$\rho_1$}%
}}}}
\end{picture}%
\caption{Non-Simplicial Fan in $\R^3$}
\label{figure:counterex_def}
\end{center}
\end{figure}

The primitive collections for this fan are:
\[
P_1 = \{\p_0, \p_1, \p_3\},\quad P_2= \{ \p_0, \p_2, \p_4\}.
\]
A first observation is that if we used Batyrev's definition of
primitive collection in this case, we would want every proper subset
of $P_1$ and $P_2$ to generate a cone of $\fan$.  This clearly isn't
true, and in fact this example has \emph{no} primitive collections if
we use Batyrev's definition.  This explains why Definition~\ref{pcdef}
is the correct definition in the non-simplicial case.

Theorem~\ref{theorem_mori_nonsimp} states that $\CPL(\fan) \subset
\PL(\fan)$ is defined by the primitive inequalities coming from the
primitive collections $P_1$ and $P_2$.  However, the proof of the
theorem shows that we need only one.  To see why, consider the
simplicial refinement $\fan'$ of $\fan$ given by subdividing
non-simplicial cone $\sigma_5$ along $\pos(\p_2,\p_4)$.  This gives
the fan pictured in Example~\ref{ex_simplicial}.  The fan $\fan'$ has
the same generators $\p_0,\dots,\p_4$ as $\fan$, and the primitive
collections for $\fan'$ are
\[
P_1' = \{\p_1, \p_3\},\quad P_2= \{ \p_0, \p_2, \p_4\}.
\]
One easily checks that $P_1'$ is of Type A and $P_2$ is of Type B and
hence is a primitive collection for $\fan$.  By \eqref{CPL_eq2},
$\CPL(\fan)$ is defined by $P_2$, so that $\phi \in \PL(\fan)$ is
convex if and only if
\[
\phi(\p_0) + \phi(\p_2) + \phi(\p_4) \ge  \phi(\p_0 + \p_2+\p_4).
\]

It is interesting to note that the Type A primitive collection $P_1' =
\{\p_1, \p_3\}$ also plays an important role.  The primitive relation
of $P_1'$ is 
\[
\p_1 + \p_3 = \p_2 + \p_4.
\]
Now take $\phi \in \PL(\fan)$.  As noted in the proof of
Theorem~\ref{theorem_mori_nonsimp}, this Type A primitive collection
gives the \emph{equality}
\begin{equation}
\label{phi_equality}
\phi(\p_1) + \phi(\p_3) = \phi(\p_1 + \p_3),
\end{equation}
which by the above primitive relation implies
\[
\phi(\p_1) + \phi(\p_3) = \phi(\p_2) + \phi(\p_4).  
\]
It is easy to see that this equality defines $\PL(\fan)$ inside of
$\PL(\fan')$.  In other words, $\phi \in \PL(\fan')$ lies in
$\PL(\fan)$ if and only if it satisfies \eqref{phi_equality} coming
from the Type A primitive collection for $\fan'$.

If we turn our attention to the other primitive collection $P_1 =
\{\p_0, \p_1, \p_3\}$ for $\fan$, then one easily sees that $\phi \in
\PL(\fan)$ is convex if and only if
\[
\phi(\p_0) + \phi(\p_1) + \phi(\p_3) \ge  \phi(\p_0 + \p_1+\p_3).
\]
This follows by considering the other simplicial refinement of $\fan$
obtained by subdividing $\sigma_5$ along $\pos(\p_1,\p_3)$.  \hfill
\qed
\end{example}

Example~\ref{nonsimpl_ex} has some interesting features:
\begin{itemize}
\item Every primitive collection for $\fan$ comes from a
Type B primitive collection for a simplicial refinement $\fan'$ of
$\fan$ satisfying $\fan'(1) = \fan(1)$.
\item For each such refinement  $\fan'$ of
$\fan$, the Type A primitive collections for $\fan'$ define $\PL(\fan)
  \subset \PL(\fan')$. 
\end{itemize}
We will see below that these properties hold in general.  

\subsection{Properties of Primitive Collections}
We begin with the following useful property of primitive collections.

\begin{proposition}
\label{thm_strict}
Let $\fan$ be a fan in $N_\R \cong \R^n$ such that $\fan$ has convex
support of dimension $n$.  If $P$ is a primitive collection for $\fan$,
then every proper subset $Q$ of $P$ is linearly independent.
\end{proposition}

\begin{proof}
We use induction on $|Q|$. If $|Q|=1$ there is nothing to show.  Now
assume that $|Q| = k+1$, $k \ge 1$, and that every $k$-element subset
of $Q$ is linearly independent.

We show that $Q$ is linearly independent by contradiction.  Hence
suppose $Q$ is linearly dependent.  Then our induction hypothesis
implies that the subspace $\mathrm{span}(Q)$ has dimension $k$.
Define $\widetilde{\fan} = \{ \sigma \cap \mathrm{span}(Q) \mid
\sigma \in \fan\}$. We omit the
straightforward proof that $\widetilde{\fan}$ is a fan in
$\mathrm{span}(Q)$.

Now fix $\p \in Q$ and let $\sigma_\p$ be the minimal cone of $\fan$
containing $P\setminus \{\p\}$.  Notice that $\sigma_\p$ does not
contain $\p$ since $P$ is a primitive collection.  Also let $\sigma_Q$
be the minimal cone of $\fan$ containing $Q$.  The cones
$\widetilde{\sigma}_{Q} = \sigma_Q \cap \mathrm{span}(Q)$ and
$\widetilde{\sigma}_\p = \sigma_\p \cap \mathrm{span}(Q)$ are in the
fan $\widetilde{\fan}$ and $\widetilde{\sigma}_Q \ne
\widetilde{\sigma}_\p$ since $\p$ is contained in
$\widetilde{\sigma}_{Q}$ but not in
$\widetilde{\sigma}_\p$. Therefore, their intersection is at most
$(k-1)$-dimensional.  On the other hand, the intersection contains $k$
linearly independent vectors
\[
Q\setminus \{\p\} \subset \widetilde{\sigma}_{Q} \cap
\widetilde{\sigma}_\p,
\]
which is a contradiction.
\end{proof}

\begin{corollary}
\label{cor_strict}
Let $\fan$ be a fan in $N_\R \cong \R^n$ such that $|\fan|$ is convex
support of dimension $n$. Then every primitive collection for $\fan$
has at most $n+1$ elements.
\end{corollary}

\begin{proof}
This follows immediately from proposition \ref{thm_strict} since any
maximal proper subset $Q = P\setminus\{\p\}$ is linearly independent
and hence has at most $n$ elements. Therefore $P = Q \cup \{ \p\}$ has
at most $n+1$ elements.
\end{proof}

\begin{remark}
Proposition~\ref{thm_strict} and Corollary~\ref{cor_strict} are
trivial in the simplicial case.
\end{remark}

\subsection{Type A Description of $\PL(\fan)$}
Let $\fan$ be a fan in $N_\R \cong \R^n$ with convex support of dimension
$n$, and let $\fan'$ be a simplicial refinement with $\fan(1) =
\fan'(1)$.  Given $\sigma \in \fan$, let $\Sigma_\sigma = \{\sigma'
\in \fan' \mid \sigma' \subset \sigma\}$.  The following convexity
result will be useful.  

\begin{lemma}
\label{lemma_subfab_convex}
Let $\sigma$ be a non-simplicial cone in $\fan$ and take an interior
wall $\tau'$ of $\fan_{\sigma}$ with $\tau=\sigma'_1 \cap \sigma'_2$,
$\sigma'_1, \sigma'_2 \in \fan_{\sigma}(n)$.  Then $\sigma'_1 \cup
\sigma'_2$ is convex.
\end{lemma}

\begin{proof}
Since $\tau'(1) \subset \sigma(1)$, $\tau'$ divides $\sigma$ into two
convex subcones $\sigma^+,\sigma^-$ with $\tau = \sigma^+ \cap
\sigma^-$.  We may assume $\sigma'_1 \subset \sigma^+,\ \sigma'_2
\subset \sigma^-$.  Given $u \in \sigma'_1, v \in \sigma'_2$, it
follows easily that the line segment $\overline{uv}$ lies in
$\sigma'_1 \cup \sigma'_2$.
\end{proof}

\begin{corollary}
\label{corollary_subfab_convex}
In the situation of Lemma~\ref{lemma_subfab_convex}, let $P$ be the
two element set 
\[
P = \sigma'_1(1) \cup \sigma'_2(1) \setminus \tau'(1).
\]
Thus $P$ consists of the generators of $\sigma'_1, \sigma'_2$ not
lying in the wall $\tau'=\sigma'_1 \cap \sigma'_2$.  Then $P$ is a
primitive collection for $\fan'$.
\end{corollary}

\begin{proof}
First note that $P$ is contained in neither $\sigma'_1$ nor
$\sigma'_2$.  Since $P$ is contained in the convex set $\sigma'_1 \cup
\sigma'_2$, it follows that $P$ is contained in no cone of $\fan'$.
Thus $P$ is a primitive collection for $\fan'$ since has it only has
two elements.
\end{proof}

As in the proof of Theorem~\ref{theorem_mori_nonsimp}, a primitive
collection for $\fan'$ has Type A when it is contained in a cone of
$\fan$.  Hence the primitive collection for $\fan'$ constructed in
Corollary~\ref{corollary_subfab_convex} has Type A.  The idea is that
these two-element primitive collections define $\PL(\fan)$ inside
$\PL(\fan')$.

\begin{proposition}
Let $\fan$ be a fan in $N_\R \cong \R^n$ with convex support of
dimension $n$ and let $\fan'$ be a simplicial refinement with
$\fan(1)=\fan'(1)$.  Then
\begin{align*}
\PL(\fan) = \big\{& \phi \in \PL(\fan') \mid \phi(\p_1 + \p_2) =
\phi(\p_1) + \phi(\p_2)\ \text{for all} \\ &\text{\rm Type A
primitive collections} \ \{ \p_1,\p_2\}\ \text{\rm for} \ \fan'
\big\}.
\end{align*}

\end{proposition}

\begin{proof}
The inclusion $\subset$ is obvious since elements of $\PL(\fan)$ are
linear on cones of $\fan$ and a Type A primitive collection is
contained in such a cone.

For the opposite inclusion, take $\phi \in \PL(\fan')$ such that
$\phi(\p_1 + \p_2) = \phi(\p_1) + \phi(\p_2)$ for all two element Type
A primitive collections for $\fan'$.  For each $\sigma' \in \fan'(n)$,
there is $m_{\sigma'} \in M_\R$ such that $\phi(u) = \langle
m_{\sigma'},u\rangle$ for $u \in \sigma'$.  It suffices to show that
$m_{\sigma_1'} = m_{\sigma_2'}$ for cones $\sigma_1', \sigma_2'$ that
lie in the same cone $\sigma$ of $\fan$ and intersect in a wall
$\sigma_1' \cap \sigma_2'=\tau'$.  This is the situation of
Corollary~\ref{corollary_subfab_convex}, where
\[
\sigma'_1(1) \cup \sigma'_2(1) = \tau'(1) \cup
\{\p_1,\p_2\}
\]
and $P = \{\p_1,\p_2\}$ is a two element Type A primitive collection
for $\fan'$.  We label the elements of $P$ so that $\p_1 \in
\sigma_1'$ and $\p_2 \in \sigma_2'$.

Since $\sigma'_1 \cup \sigma'_2$ is convex by
Lemma~\ref{lemma_subfab_convex}, it contains $\p_1+\p_2$.  We may
assume $\p_1+\p_2 \in \sigma'_2$ without loss of generality.  Then
\[
\langle m_{\sigma_1'},\p_1\rangle = \phi(\p_1) = - \phi(\p_2)+
\phi(\p_1 + \p_2) = -\langle m_{\sigma_2'},\p_1\rangle
+ \langle m_{\sigma_2'},\p_1+p_2\rangle = \langle
m_{\sigma_2'},\p_1\rangle.
\]
Since $m_{\sigma_1'} - m_{\sigma_2'} \in \tau^{\prime\perp}$, it
follows that $m_{\sigma_1'} = m_{\sigma_2'}$.  This completes the proof.
\end{proof}

\subsection{Primitive Collections Supported on Simplicial Refinements}
In the fan $\fan$ pictured in Figure~\ref{figure:counterex_def} in
Example~\ref{nonsimpl_ex}, we saw that every primitive collection for
$\fan$ came from a primitive collection for a simplicial subdivision
of $\fan$.

In general, if $\fan'$ is a simplicial subdivision of $\fan$ with
$\fan'(1) = \fan(1)$, we say that a primitive collection $P$ for
$\fan$ is \emph{supported} on $\fan'$ if $P$ is also a primitive
collection for $\fan'$.  We now prove that all primitive collections
for $\fan$ are supported on such simplicial subdivisions.  Here is the
precise result.

\begin{proposition} 
\label{prim_lemma} 
Let $\fan$ be a fan in $N_\R \cong \R^n$ with convex support of
dimension $n$ and let $P$ be a primitive collection for $\fan$.  Then
there exists a simplicial refinement $\fan'$ with $\fan'(1) = \fan(1)$
such that $P$ is a primitive collection for $\fan'$.  Furthermore, if
$\fan$ is quasi-projective, then $\fan'$ can be chosen to be
quasi-projective.
\end{proposition}

\begin{proof}
By Proposition~\ref{thm_strict}, every proper subset $P$ is linearly
independent.  In particular, if $\sigma \in \fan$, then
$P\cap\sigma(1)$ is a proper subset of $P$ (since $P$ is a primitive
collection) and hence is linearly independent.  Thus we can apply
Theorem~\ref{thrm_simpl_ref_general} to obtain a simplicial refinement
$\fan'$ such that $P\cap\sigma(1)$ generates a cone of $\fan'$ for all
$\sigma \in \fan$.  The theorem also allows us to assume $\fan'$ is
quasi-projective whenever $\fan$ is.

We claim that $P$ is a primitive collection for $\fan'$.  First note
that if $P$ were contained in a cone of $\fan'$, then it would be
contained in a cone of $\fan$, which we know to be false.  Now let $Q$
be a proper subset of $P$.  Then $Q$ is contained in a cone $\sigma
\in \fan$, so that $Q \subset P\cap\sigma(1)$.  Since $P\cap\sigma(1)$
generates a cone of $\fan'$, it follows that $Q$ is contained in a
cone of $\fan'$.  Hence $P$ is a primitive collection for $\fan'$.
\end{proof}

\begin{remark}
When $\fan$ is non-simiplicial, it may be impossible to find a
\emph{single} simplicial refinement $\fan'$ such that \emph{every}
primitive collection for $\fan$ is also primitive for $\fan'$.  In
Figure \ref{figure:counterex_def} from Example~\ref{nonsimpl_ex}, we
see two primitive collections $P_1 = \{\p_0, \p_1, \p_3\}$ and $P_2 =
\{ \p_0, \p_2, \p_4\}$, but there is no simplicial refinement $\fan'$
of $\fan$ with $\fan'(1) = \fan(1)$ that supports both $P_1$ and
$P_2$.
\end{remark}

\section{Is Quasi-Projective Necessary?}

In this section we explore an open question about primitive
collections.  In \cite{casagrande}, Casagrande raises the question of
whether $\CPL(\fan)$ is defined by primitive inequalities when $\fan$
is not quasi-projective.  Here is a classic example.

\begin{example}
\label{ex_simplicial_nonqp}
The following example of a non-projective smooth complete fan is taken
from Fulton \cite[p.\ 71]{fulton}.  Consider the fan
$\fan$ in $\R^3$ with seven minimal generators:
\begin{align*}
&\p_1=(-1,0,0), \ \p_2=(0,-1,0), \ \p_3=(0,0,-1), \ \p_4=(1,1,1), \\
&\p_5=(1,1,0), \ p_6=(0,1,1), \ \p_7=(1,0,1).
\end{align*}
The cones of $\fan$ are obtained by projecting from the origin through
the triangulated polytope shown in
Figure~\ref{figure:ex_simplicial_nonqp}.  The fan $\fan$ has 15 walls
and 10 maximal cones.
\begin{figure}[htb]
\begin{center}
\begin{picture}(0,0)%
\includegraphics{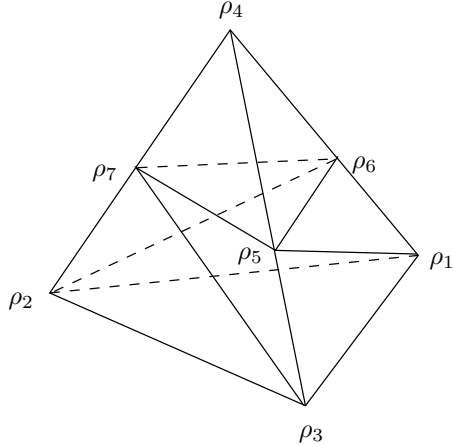}%
\end{picture}%
\setlength{\unitlength}{4144sp}%
\begingroup\makeatletter\ifx\SetFigFont\undefined%
\gdef\SetFigFont#1#2#3#4#5{%
  \reset@font\fontsize{#1}{#2pt}%
  \fontfamily{#3}\fontseries{#4}\fontshape{#5}%
  \selectfont}%
\fi\endgroup%
\begin{picture}(2519,2700)(1149,-2543)
\put(3668,-1463){\makebox(0,0)[lb]{\smash{{\SetFigFont{10}{14.4}{\rmdefault}{\mddefault}{\updefault}{\color[rgb]{0,0,0}$\rho_1$}%
}}}}
\put(2888,-2498){\makebox(0,0)[lb]{\smash{{\SetFigFont{10}{14.4}{\rmdefault}{\mddefault}{\updefault}{\color[rgb]{0,0,0}$\rho_3$}%
}}}}
\put(3203,-892){\makebox(0,0)[lb]{\smash{{\SetFigFont{10}{14.4}{\rmdefault}{\mddefault}{\updefault}{\color[rgb]{0,0,0}$\rho_6$}%
}}}}
\put(1149,-1703){\makebox(0,0)[lb]{\smash{{\SetFigFont{10}{14.4}{\rmdefault}{\mddefault}{\updefault}{\color[rgb]{0,0,0}$\rho_2$}%
}}}}
\put(2408, 37){\makebox(0,0)[lb]{\smash{{\SetFigFont{10}{14.4}{\rmdefault}{\mddefault}{\updefault}{\color[rgb]{0,0,0}$\rho_4$}%
}}}}
\put(1651,-938){\makebox(0,0)[lb]{\smash{{\SetFigFont{10}{14.4}{\rmdefault}{\mddefault}{\updefault}{\color[rgb]{0,0,0}$\rho_7$}%
}}}}
\put(2520,-1432){\makebox(0,0)[lb]{\smash{{\SetFigFont{10}{14.4}{\rmdefault}{\mddefault}{\updefault}{\color[rgb]{0,0,0}$\rho_5$}%
}}}}
\end{picture}%
\end{center}
\caption{Non Quasi-Projective Example}
\label{figure:ex_simplicial_nonqp}
\end{figure}

The seven primitive collections for $\fan$ and their associated primitive
relations are:
\begin{align*}
\{\p_2,\p_4\}\colon & \p_2+\p_4 = \p_7\\
\{\p_1,\p_4\}\colon &\p_1 + \p_4 =\p_6\\
\{\p_2,\p_5\}\colon &\p_2+\p_5= \p_3+\p_7\\
\{\p_3,\p_6\}\colon &\p_3+\p_6=\p_1+\p_5\\
\{\p_3,\p_4\}\colon &\p_3+\p_4=\p_5\\
\{\p_1,\p_7\}\colon &\p_1+\p_7=\p_2+\p_6\\
\{\p_5,\p_6,\p_7\}\colon &\p_5+\p_6+\p_7=2\p_4.
\end{align*}
By \eqref{cone_inclusion}, a convex function $\phi \in \CPL(\fan)$
satisfies the primitive inequalities:
\begin{equation}
\label{fulton_ineq}
\begin{aligned}
&\phi(\p_2)+\phi(\p_4) \geq \phi(\p_7)\\
&\phi(\p_1) + \phi(\p_4) \geq\phi(\p_6)\\
&\phi(\p_5)+\phi(\p_2)\geq \phi(\p_3)+\phi(\p_7)\\
&\phi(\p_3)+\phi(\p_6)\geq\phi(\p_1)+\phi(\p_5)\\
&\phi(\p_3)+\phi(\p_4)\geq\phi(\p_5)\\
&\phi(\p_1)+\phi(\p_7)\geq\phi(\p_2)+\phi(\p_6)\\
&\phi(\p_5)+\phi(\p_6)+\phi(\p_7)\geq2\phi(\p_4).
\end{aligned}
\end{equation}
Notice that adding up the third, fourth and sixth inequalities yields an
equality, hence we have 3 equalities:
\begin{align*}
&\phi(\p_2)+\phi(\p_5)=  \phi(\p_3)+\phi(\p_7)\\
&\phi(\p_3)+\phi(\p_6)=\phi(\p_1)+\phi(\p_5)\\
&\phi(\p_1)+\phi(\p_7)=\phi(\p_2)+\phi(\p_6).
\end{align*}

To see what this says about the nef cone $\nef(X)$, note that
\[
\nef(X) \cong \{\phi \in \CPL(\fan) \mid \phi(\p_1) = \phi(\p_2) =
\phi(\p_3) = 0\}.
\]
Assume $\phi(\p_1) = \phi(\p_2) = \phi(\p_3) = 0$. Then the three
equalities give $\phi(\p_5) = \phi(\p_6) = \phi(\p_7)$.  Define $a=
\phi(\p_4)$ and $b= \phi(\p_5) = \phi(\p_6) = \phi(\p_7)$. Then
inequalities \eqref{fulton_ineq} imply $a \ge b$ and $3b \ge 2a$. It
follows that $\nef(X)$ is contained in the 2-dimensional cone pictured
in Figure~\ref{figure:ex_simpl_nqp_nef}.
\begin{figure}[htb]
\begin{center}
\begin{picture}(0,0)%
\includegraphics{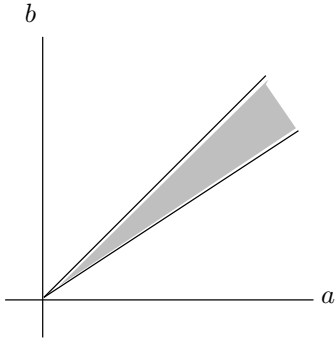}%
\end{picture}%
\setlength{\unitlength}{4144sp}%
\begingroup\makeatletter\ifx\SetFigFont\undefined%
\gdef\SetFigFont#1#2#3#4#5{%
  \reset@font\fontsize{#1}{#2pt}%
  \fontfamily{#3}\fontseries{#4}\fontshape{#5}%
  \selectfont}%
\fi\endgroup%
\begin{picture}(2163,2010)(1789,-1873)
\put(3691,-1658){\makebox(0,0)[lb]{\smash{{\SetFigFont{10}{14.4}{\familydefault}{\mddefault}{\updefault}{\color[rgb]{0,0,0}$a$}%
}}}}
\put(1921, 29){\makebox(0,0)[lb]{\smash{{\SetFigFont{10}{14.4}{\familydefault}{\mddefault}{\updefault}{\color[rgb]{0,0,0}$b$}%
}}}}
\put(3541,-458){\makebox(0,0)[lb]{\smash{{\SetFigFont{12}{14.4}{\familydefault}{\mddefault}{\updefault}{\color[rgb]{0,0,0}}%$\mathrm{Nef}(X)$
}}}}
\end{picture}%
\end{center}
\caption{Cone Defined by Primitive Inequalities}
\label{figure:ex_simpl_nqp_nef}
\end{figure}

Since $\pic(X)_\R$ has dimension 4 and $\nef(X)$ has dimension at most
two, we see that $X$ is non-projective since the nef cone does not
have maximal dimension.

It is also easy to see that the cone in
Figure~\ref{figure:ex_simpl_nqp_nef} actually equals the nef cone
$\nef(X)$---just show that the generators of this cone are nef.  For
example, when $a = b > 0$, note that $\fan$ is a refinement of the
complete fan $\fan_0$ with 1-dimensional generators
$\p_1,\p_2,\p_3,\p_4$.  The toric variety of $\fan_0$ is $\P^3$, and
the class corresponding to $a=b> 0$ is the pullback of an ample
divisor on $\P^3$, hence nef on $X$.  For $3b = 2a > 0$, one proceeds
similarly by noting that $\fan$ is a refinement of the projective
non-simplical fan $\fan_1$ with 1-dimensional generators
$\p_1,\p_2,\p_3,\p_5,\p_6,\p_7$.  \hfill\qed
\end{example}

Other more substantial examples can be found in Chapter 7 of
Scaramuzza's thesis \cite{scaramuzza}.  Based on this, we make the
following conjecture, which we credit to Casagrande.

\begin{conjecture}[Casagrande]
\label{casagrande}
Let $X$ be a simplicial toric variety coming from the fan
$\fan$ in $N_\R \cong \R^n$.  If $|\fan|$ is convex of
dimension $n$, then
\begin{align*}
\CPL(\Sigma) = \big\{ &\phi \in \PL(\fan) \mid \phi(\p_1) +
\cdots + \phi(\p_k) \ge \phi(\p_1 + \cdots + \p_k )\\ 
&\text{\rm for all primitive collections}\ \{ \p_1,\dots, \p_k \}
\ \text{\rm for}\ \fan \big\}. 
\end{align*}
\end{conjecture}

Besides the evidence provided by examples, we also have the
theoretical result of Casagrande \cite[Thm.\ 5.6]{casagrande}, which
states that if a smooth complete non-projective toric variety $X$ has
a toric blow-up $Y \to X$ with $Y$ projective, then
Conjecture~\ref{casagrande} holds for $X$.

We stated Conjecture~\ref{casagrande} for the simplicial case because
of the following result.

\begin{proposition}
\label{if_conjecture}
Assume that Conjecture~\ref{casagrande} is true.  Let $X$ be a
non-simplicial toric variety of a fan $\fan$ in $N_\R \cong \R^n$ such
that $|\fan|$ is convex of dimension $n$.  Then:
\begin{align*}
\CPL(\Sigma) = \big\{ &\phi \in \PL(\fan) \mid \phi(\p_1) +
\cdots + \phi(\p_k) \ge \phi(\p_1 + \cdots + \p_k )\\ 
&\text{\rm for all primitive collections}\ \{ \p_1,\dots, \p_k \}
\ \text{\rm for}\ \fan \big\}. 
\end{align*}
Furthemore, every primitive collection for $\fan$ is supported on a
simplicial refinement $\fan'$ of $\fan$ satisfying $\fan'(1) =
\fan(1)$.
\end{proposition}

\begin{proof}
The first part of the proposition follows since the proof of
Theorem~\ref{theorem_mori_nonsimp} (the non-simplicial case of our
main theorem) requires the existence of a simplicial refinement
$\fan'$ of $\fan$ with the following properties:
\begin{itemize}
\item $\fan'(1) = \fan(1)$. 
\item $\CPL(\fan')$ is described using primitive inequalities.
\end{itemize}
If we assume Conjecture~\ref{casagrande}, then the second bullet is
automatically true, which means that
Corollary~\ref{cor_simpl_ref_general} gives the needed simplicial
subdivision of $\fan$.

For the final assertion, observe that the proof of
Proposition~\ref{prim_lemma} applies without change since the first
part of Theorem~\ref{thrm_simpl_ref_general} gives the required
simplicial refinement of $\fan$ without needing to assume
quasi-projective.
\end{proof}

One way to think about Conjecture~\ref{casagrande} is that once this
conjecture is proved, the results of this paper would apply to any fan
in $N_\R \cong \R^n$ whose support is convex of dimension $n$---there
would be no requirement that $\fan$ be quasi-projective.  However,
the proofs of the simplicial case given in
Theorem~\ref{theorem_mori_primitive} make essential use of extremal
rays, which exist only in the quasi-projective case.  The result of
Casagrande \cite[Thm.~5.6]{casagrande} mentioned above is a good first
step, but it is likely that some significantly new ideas will be
needed to prove Conjecture~\ref{casagrande} in general.

\section*{Acknowledgements} This paper is based on the second
author's PhD thesis at the University of Massachusetts
\cite{vonRenesse}, written under the direction of Eduardo Cattani and
the first author.  Both authors are very grateful to Eduardo Cattani
for many stimulating discussions.  The first author would also like to
thank Anna Scaramuzza for asking some interesting questions about
primitive collections.  We would also like to thank Cinzia Casagrande
for bringing \cite{kresch} to our attention and for giving a
counterexample to a conjecture in the original version of the paper.

\address{% First Author
Department of Mathematics \\ and Computer Science \\
Amherst College\\
Amherst, MA 01002-5000 \\
USA
}
{dac@cs.amherst.edu}
%%%%%%%%%
\address{% Second Author
Department of Mathematics \\
Westfield State College \\
Westfield, MA 01086 \\
USA
}
{cvonrenesse@wsc.ma.edu}

\end{document}